\documentclass[twoside]{article}
\usepackage{mormodified}

\received{}                              %
\revised{}                               %
\pubyear{0x}                             %
\pubmonth{Xxxxxxx}                       %
\volume{xx}                              %
\issue{x}                                %
\pages{xxx--xxx}                         %
\firstpage{0xxx}                         %
\DOI{10.1287/moor.xxxx.xxxx}             %
\startpagenumber{1}                      %

\title{Robust Sensitivity Analysis for Stochastic Systems}
\ShortTitle{Robust Sensitivity Analysis for Stochastic Systems}
\ShortAuthors{Lam}
\NumberOfAuthors{1}
\FirstAuthor{Henry Lam}
\FirstAuthorAddress{Department of Industrial and Operations Engineering, University of Michigan, Ann Arbor}
\FirstAuthorEmail{khlam@umich.edu}
\FirstAuthorURL{}
\SecondAuthor{}
\SecondAuthorAddress{}
\SecondAuthorEmail{}
\SecondAuthorURL{}
\ThirdAuthor{}
\ThirdAuthorAddress{}
\ThirdAuthorEmail{}
\ThirdAuthorURL{}
\FourthAuthor{}
\FourthAuthorAddress{}
\FourthAuthorEmail{}
\FourthAuthorURL{}
\FifthAuthor{}
\FifthAuthorAddress{}
\FifthAuthorEmail{}
\FifthAuthorURL{}


\MSCcodes{Primary: 65C05, 62G35; Secondary: 90C31, 46G05} 

\ORMScodes{Primary: decision analysis--sensitivity, simulation--statistical analysis, statistics--nonparametric; Secondary: simulation--efficiency} 

\begin{document}
\maketitle


\begin{abstract}

We study a worst-case approach to measure the sensitivity to model misspecification in the performance analysis of stochastic systems. The situation of interest is when only minimal parametric information is available on the form of the true model. Under this setting, we post optimization programs that compute the worst-case performance measures, subject to constraints on the amount of model misspecification measured by Kullback-Leibler (KL) divergence. Our main contribution is the development of infinitesimal approximations for these programs, resulting in asymptotic expansions of their optimal values in terms of the divergence. The coefficients of these expansions can be computed via simulation, and are mathematically derived from the representation of the worst-case models as changes of measure that satisfy a well-defined class of functional fixed point equations.

\end{abstract}
\normalsize


\section{Introduction}
Any performance analysis of stochastic systems requires model assumptions that, to various extent, deviate from the truth. 
Understanding how these model errors affect the analysis is of central importance in stochastic modeling.




This paper concerns the robust approach to measure the impacts of model errors: 
Given a baseline model that is believed to reasonably approximate the truth, without any specific information on how it is misrepresented,
an optimization program is imposed to evaluate the worst-case performance measure among all models that are close to the baseline in the sense of some nonparametric statistical distance, such as Kullback-Leibler (KL).





The main contribution of this paper is to bring in a new line of infinitesimal analysis for the worst-case optimization described above. Namely, taking the viewpoint that the true model is within a small neighborhood of the baseline model, we conduct an asymptotic expansion on the worst-case objective value as the statistical distance that defines the neighborhood shrinks to zero. The primary motivation for this asymptotic analysis is to handle the difficulty in direct solution of these worst-case optimizations in the context of stochastic systems driven by standard i.i.d.~input processes (being non-convex and infinite-dimensional). In particular, the coefficients of our expansions are computable via simulation, hence effectively converting the otherwise intractable optimizations into simulation problems. This approach thus constitutes a tractable framework for nonparametric sensitivity analysis as the expansion coefficients capture the worst-case effect on a performance measure when the model deviates from the truth in the nonparametric space.

\section{Formulation and Highlights}\label{sec:idea}
Define a \emph{performance measure} $E[h(\mathbf{X}_T)]$, where $h(\cdot)$ is a real-valued \emph{cost function}, $\mathbf{X}_T=(X_1,X_2,\ldots,X_T)$ is a sequence of i.i.d.~random objects each lying on the domain $\mathcal{X}$, and $T$ is the time horizon. We assume that the cost function $h$ can be evaluated given its argument, but does not necessarily have closed form. For example, $h(\mathbf{X}_T)$ can be the waiting time of the $100$-th customer in a queueing system, where $\mathbf{X}_T$ is the sequence of interarrival and service time pairs.

Our premise is that there is a \emph{baseline} model that is believed to approximately describe each i.i.d.~$X_t$. The probability distribution that governs the baseline model is denoted $P_0$. Correspondingly, the baseline performance measure is $E_0[h(\mathbf{X}_T)]$, where $E_0[\cdot]$ is the expectation under the product measure $P_0^T=P_0\times P_0\times\cdots\times P_0$. On the other hand, we denote $P_f$ as the distribution that governs the true model (which is unknown), and analogously, $E_f[\cdot]$ as the expectation under the product measure $P_f^T$.

We are interested in the worst (or best)-case optimizations
\begin{equation}
\begin{array}{ll}
\max&E_f[h(\mathbf{X}_T)]\\
\text{subject to}\ &D(P_f\|P_0)\leq\eta\\
&X_t\stackrel{i.i.d.}{\sim}P_f\text{\ \ for\ }t=1,\ldots,T\\
&P_f\in\mathcal P_0
\end{array} \label{max fixed time}
\end{equation}
and
\begin{equation}
\begin{array}{ll}
\min&E_f[h(\mathbf{X}_T)]\\
\text{subject to}\ &D(P_f\|P_0)\leq\eta\\
&X_t\stackrel{i.i.d.}{\sim}P_f\text{\ \ for\ }t=1,\ldots,T\\
&P_f\in\mathcal P_0.
\end{array} \label{min fixed time}
\end{equation}

Here $P_f$ is the decision variable. The space $\mathcal P_0$ denotes the set of all distributions absolutely continuous with respect to the baseline $P_0$. The constraint $D(P_f\|P_0)\leq\eta$ represents the $\eta$-neighborhood surrounding $P_0$, using KL divergence as the notion of distance, i.e.
$$D(P_1\|P_2):=\int\log\frac{dP_1}{dP_2}dP_1=E_2\left[\frac{dP_1}{dP_2}\log\frac{dP_1}{dP_2}\right]$$
where $(dP_1/dP_2)$ is the likelihood ratio, equal to the Radon-Nikodym derivative of $P_1$ with respect to $P_2$, and $E_2[\cdot]$ is the expectation under $P_2$. In brief, the pair of optimizations \eqref{max fixed time} and \eqref{min fixed time} describes the most extreme performance measures among any $P_f$ within $\eta$ units of KL divergence from the baseline $P_0$.

Note that we use the notation $X_t\stackrel{i.i.d.}{\sim}P_f$ in \eqref{max fixed time} and \eqref{min fixed time} to highlight the assumption that $X_t$'s are i.i.d.~each with distribution $P_f$. This i.i.d.~property deems \eqref{max fixed time} and \eqref{min fixed time} non-convex and difficult to solve in general.

The major result of this paper stipulates that when letting $\eta$ go to 0, under mild assumptions on $h$, the optimal values of \eqref{max fixed time} and \eqref{min fixed time} can each be expressed as
\begin{equation}
\max/\min E_{f}[h(\mathbf{X}_T)]=E_0[h(\mathbf{X}_T)]+\zeta_1(P_0,h)\sqrt\eta+\zeta_2(P_0,h)\eta+\cdots \label{expansion}
\end{equation}
where $\zeta_1(P_0,h),\zeta_2(P_0,h),\ldots$ is a sequence of coefficients that can be written explicitly in terms of $h$ and $P_0$.

To avoid redundancy, in the following discussion we will focus on the maximization formulation, and will then point out the adaptation to minimization formulation briefly.

%


\section{Main Results}\label{sec:results}

\subsection{Single-Variable Case}\label{section:basic}
Consider the special case $T=1$, namely the cost function $h$ depends only on a single variable $X\in\mathcal{X}$ in formulation \eqref{max fixed time}:
\begin{equation}
\begin{array}{ll}
\max&E_f[h(X)]\\
\text{subject to}\ &D(P_f\|P_0)\leq\eta\\
&X\sim P_f\\
&P_f\in\mathcal{P}_0.
\end{array} \label{max basic}
\end{equation}
The result for this particular case will constitute an important building block in our later development. 

We make two assumptions. First, for $X\sim P_0$, we impose a finite exponential moment condition on $h(X)$:
\begin{assumption}
The variable $h(X)$ has finite exponential moment in a neighborhood of 0 under $P_0$, i.e.~$E_0[e^{\theta h(X)}]<\infty$ for $\theta\in(-r,r)$ for some $r>0$. 
\label{exponential moment}
\end{assumption}

Second, we impose the following non-degeneracy condition:
\begin{assumption}
The variable $h(X)$ is non-constant under $P_0$. \label{nondegeneracy}
\end{assumption}

The first assumption on the light-tailedness of $h(X)$ is in particular satisfied by any bounded $h(\mathbf X_T)$, which handles all probability estimation problems for instance. The second assumption ensures that the baseline distribution $P_0$ is not a ``locally optimal" model, in the sense that there always exists an opportunity to upgrade the value of the performance measure by rebalancing the probability measure.

Under the above assumptions, we can get a very precise understanding of the objective value when $\eta$ is small:
\begin{thm}
Let $T=1$ in formulation \eqref{max fixed time}, with $h(\cdot):\mathcal X\to\mathbb R$. Suppose Assumptions \ref{exponential moment} and \ref{nondegeneracy} hold. Denote $\psi(\beta)=\log E_0[e^{\beta h(X)}]$ as the logarithmic moment generating function of $h(X)$. When $\eta>0$ is within a sufficiently small neighborhood of 0, the optimal value of \eqref{max fixed time} is given by
\begin{equation}
\max E_f[h(X)]=\psi'(\beta^*) \label{full relation}
\end{equation}
where $\beta^*$ is the unique positive solution to the equation $\beta\psi'(\beta)-\psi(\beta)=\eta$. This implies
\begin{equation}
\max E_f[h(X)]=E_0[h(X)]+\sqrt{2Var_0(h(X))}\eta^{1/2}+\frac{1}{3}\frac{\kappa_3(h(X))}{Var_0(h(X))}\eta+O(\eta^{3/2}) \label{basic expansion}
\end{equation}
where $Var_0(h(X))$ and $\kappa_3(h(X))$ are the variance and the third order cumulant of $h(X)$ under $P_0$ respectively, i.e.
$$Var_0(h(X))=E_0[(h(X)-E_0[h(X)])^2]$$
$$\kappa_3(h(X))=E_0[(h(X)-E_0[h(X)])^3].$$
\label{thm:basic}
\end{thm}

We shall explain how to obtain Theorem \ref{thm:basic}. 
The first step is to transform the decision variables from the space of measures to the space of functions. Recall that $P_f$ is assumed to be absolutely continuous with respect to $P_0$, and hence the likelihood ratio $L:=dP_f/dP_0$ exists. Via a change of measure, the optimization problem \eqref{max basic} can be rewritten as a maximization over the likelihood ratios, i.e.~\eqref{max basic} is equivalent to
\begin{equation}
\begin{array}{ll}
\max&E_0[h(X)L(X)]\\
\text{subject to}\ &E_0[L(X)\log L(X)]\leq\eta\\
&L\in\mathcal{L}\end{array} \label{max L basic}
\end{equation}
where $\mathcal{L}:=\{L\in\mathcal{L}_1(P_0):E_0[L]=1,\ L\geq0\text{\ a.s.\ }\}$, and we denote $\mathcal{L}_1(P_0)$ as the $\mathcal{L}_1$-space with respect to the measure $P_0$ (we sometimes suppress the dependence of $X$ in $L=L(X)$ for convenience when no confusion arises). The key now is to find an optimal solution $L^*$, and investigate its asymptotic relation with $\eta$.
To this end, consider the Lagrangian relaxation
\begin{equation}
\max_{L\in\mathcal{L}}E_0[h(X)L]-\alpha(E_0[L\log L]-\eta) \label{max dual basic inner}
\end{equation}
where $\alpha$ is the Lagrange multiplier. The solution of \eqref{max dual basic inner} is characterized by the following proposition:

\begin{proposition}
Under Assumption \ref{exponential moment}, when $\alpha>0$ is sufficiently large, there exists a unique optimizer of \eqref{max dual basic inner} given by
\begin{equation}
L^*(x)=\frac{e^{h(x)/\alpha}}{E_0[e^{h(X)/\alpha}]}. \label{basic solution}
\end{equation}
\label{optimal L}
\end{proposition}

This result is known (e.g.~\cite{hs08}, \cite{pjd00}); for completeness we provide a proof in the appendix. With this proposition, we can prove Theorem \ref{thm:basic}:
\begin{proofof}{Theorem \ref{thm:basic}}
By the sufficiency result in Chapter 8, Theorem 1 in \cite{luenberger69} (shown in Theorem \ref{nonconvex} in the Appendix), suppose that we can find $\alpha^*\geq0$ and $L^*\in\mathcal L$ such that $L^*$ maximizes \eqref{max dual basic inner} for $\alpha=\alpha^*$ and $E_0[L^*\log L^*]=\eta$, then $L^*$ is the optimal solution for \eqref{max L basic}. We will show later that when $\eta$ is close to 0, we can indeed obtain such $\alpha^*$ and $L^*$. For now, assuming that such $\alpha^*$ and $L^*$ exist and that $\alpha^*$ is sufficiently large, the proof of \eqref{full relation} is divided into the following two steps:
%
%
\\

\noindent\textbf{Relation between $\eta$ and $\alpha^*$.} 
By Proposition \ref{optimal L}, $L^*$ satisfies \eqref{basic solution} with $\alpha=\alpha^*$. We have
\begin{align}
\eta&=E_0[L^*\log L^*]=\frac{E_0[h(X)L^*]}{\alpha^*}-\log E_0[e^{h(X)/\alpha^*}] \notag\\
&=\frac{\beta^*E_0[h(X)e^{\beta^*h(X)}]}{E_0[e^{\beta^*h(X)}]}-\log E_0[e^{\beta^*h(X)}]=\beta^*\psi'(\beta^*)-\psi(\beta^*) \label{inversion}
\end{align}
where we define $\beta^*=1/\alpha^*$, and $\psi(\beta)=\log E_0[e^{\beta h(X)}]$ is the logarithmic moment generating function of $h(X)$.
\\

\noindent\textbf{Relation between the optimal objective value and $\alpha^*$.} The optimal objective value is
\begin{equation}
E_0[h(X)L^*]=\frac{E_0[h(X)e^{h(X)/\alpha^*}]}{E_0[e^{h(X)/\alpha^*}]}=\frac{E_0[h(X)e^{\beta^*h(X)}]}{E_0[e^{\beta^*h(X)}]}=\psi'(\beta^*) \label{primal expansion}
\end{equation}

This gives the form in \eqref{full relation}. We are yet to show the existence of a sufficiently large $\alpha^*>0$ such that the corresponding $L^*$ in \eqref{basic solution} satisfies $E_0[L^*\log L^*]=\eta$. To this end,
we use Taylor's expansion to write
\begin{align}
\beta\psi'(\beta)-\psi(\beta)&=\sum_{n=0}^\infty\frac{1}{n!}\kappa_{n+1}\beta^{n+1}-\sum_{n=0}^\infty\frac{1}{n!}\kappa_n\beta^n \notag\\
&=\sum_{n=1}^\infty\left[\frac{1}{(n-1)!}-\frac{1}{n!}\right]\kappa_n\beta^n=\sum_{n=2}^\infty\frac{1}{n(n-2)!}\kappa_n\beta^n \notag\\
&=\frac{1}{2}\kappa_2\beta^2+\frac{1}{3}\kappa_3\beta^3+\frac{1}{8}\kappa_4\beta^4+O(\beta^5) \label{inversion 1}
\end{align}
where $\kappa_n=\psi^{(n)}(0)$ is the $n$-th cumulant of $h(X)$ under $P_0$, and the remainder $O(\beta^5)$ is continuous in $\beta$. By Assumption \ref{nondegeneracy}, we have $\kappa_2>0$. Thus for small enough $\eta$, \eqref{inversion 1} reveals that there is a small $\beta^*>0$ that is a root to the equation $\eta=\beta\psi'(\beta)-\psi(\beta)$. Moreover, this root is unique. This is because by Assumption \ref{nondegeneracy}, $\psi(\cdot)$ is strictly convex, and hence $(d/d\beta)(\beta\psi'(\beta)-\psi(\beta))=\beta\psi''(\beta)>0$ for $\beta>0$, so that $\beta\psi'(\beta)-\psi(\beta)$ is strictly increasing.

Since $\alpha^*=1/\beta^*$, this shows that for any sufficiently small $\eta$, we can find a large $\alpha^*>0$ such that the corresponding $L^*$ in \eqref{basic solution} satisfies \eqref{inversion}, or in other words $E_0[L^*\log L^*]=\eta$.

Next, using \eqref{inversion 1}, we can invert the relation
$$\eta=\frac{1}{2}\kappa_2{\beta^*}^2+\frac{1}{3}\kappa_3{\beta^*}^3+\frac{1}{8}\kappa_4{\beta^*}^4+O({\beta^*}^5)$$
to get
\begin{align*}
\beta^*&=\sqrt{\frac{2\eta}{\kappa_2}}\left(1+\frac{2}{3}\frac{\kappa_3}{\kappa_2}\beta^*+\frac{1}{4}\frac{\kappa_4}{\kappa_2}{\beta^*}^2+O({\beta^*}^3)\right)^{-1/2}\\
&=\sqrt{\frac{2\eta}{\kappa_2}}\left(1-\frac{1}{3}\frac{\kappa_3}{\kappa_2}\beta^*+O({\beta^*}^2)\right)\\
&=\sqrt{\frac{2}{\kappa_2}}\eta^{1/2}-\frac{2}{3}\frac{\kappa_3}{\kappa_2^2}\eta+O(\eta^{3/2}).
\end{align*}
As a result, \eqref{primal expansion} can be expanded as
\begin{align}
E_0[h(X)L^*]&=\psi'(\beta^*)=\kappa_1+\kappa_2\beta^*+\kappa_3\frac{{\beta^*}^2}{2}+O({\beta^*}^3) \notag\\
&=\kappa_1+\kappa_2\left(\sqrt{\frac{2}{\kappa_2}}\eta^{1/2}-\frac{2}{3}\frac{\kappa_3}{\kappa_2^2}\eta+O(\eta^{3/2})\right)+\frac{\kappa_3}{2}\left(\frac{2}{\kappa_2}\eta+O(\eta^{3/2})\right)+O(\eta^{3/2}) \notag\\
&=\kappa_1+\sqrt{2\kappa_2}\eta^{1/2}+\frac{1}{3}\frac{\kappa_3}{\kappa_2}\eta+O(\eta^{3/2}) \notag
\end{align}
which gives \eqref{basic expansion}.

\end{proofof}

\subsection{Finite Horizon Problems}
We now state our main result on formulation \eqref{max fixed time} for $T>1$. This requires correspondences of Assumptions \ref{exponential moment} and \ref{nondegeneracy}. The finite exponential moment condition is now stated as follows:
\begin{assumption}
The cost function $h$ satisfies $|h(\mathbf{X}_T)|\leq\sum_{t=1}^T\Lambda_t(X_t)$ for some deterministic functions $\Lambda_t(\cdot):\mathcal X\to\mathbb R$, where each of the $\Lambda_t(X_t)$'s possesses finite exponential moment under $P_0$, i.e.~$E_0[e^{\theta \Lambda_t(X)}]<\infty$ for $\theta$ in a neighborhood of zero. \label{boundedness}
\end{assumption}
To state our second assumption, we introduce a function $g(\cdot):=\mathcal{G}(h)(\cdot)$ where $\mathcal{G}$ is a functional acted on $h$ and $g:=\mathcal{G}(h)$ maps from $\mathcal{X}$ to $\mathbb{R}$. This function $g(x)$ is defined as the sum of individual conditional expectations of $h(\mathbf{X}_T)$ over all time steps, i.e.
\begin{equation}
g(x)=\sum_{t=1}^Tg_t(x) \label{g definition}
\end{equation}
where $g_t(x)$ is the individual conditional expectation at time $t$, given by
\begin{equation}
g_t(x)=E_0[h(\mathbf{X}_T)|X_t=x]. \label{gt}
\end{equation}
Our second assumption is a non-degeneracy condition imposed on the random variable $g(X)$ for $X\sim P_0$:
\begin{assumption}
The random variable $g(X)$ is non-constant under $P_0$.
\label{not sup}
\end{assumption}

The following is our main result:
\begin{thm}
Under Assumptions \ref{boundedness} and \ref{not sup}, the optimal value of \eqref{max fixed time} satisfies
\begin{equation}
\max E_f[h(\mathbf{X}_T)]=E_0[h(\mathbf{X}_T)]+\sqrt{2Var_0(g(X))}\eta^{1/2}+\frac{1}{Var_0(g(X))}\left(\frac{1}{3}\kappa_3(g(X))+\nu\right)\eta+O(\eta^{3/2}) \label{main expansion fixed time}
\end{equation}
where $Var_0(g(X))$ and $\kappa_3(g(X))$ are the variance and the third order cumulant of $g(X)$ respectively, 
and
\begin{equation}
\nu=E_0[(G(X,Y)-E_0[G(X,Y)])(g(X)-E_0[g(X)])(g(Y)-E_0[g(Y)])]. \label{nu}
\end{equation}
Here $g(\cdot)$ is defined in \eqref{g definition} and \eqref{gt},
and $G(\cdot,\cdot)$ is a function derived from $h$ that is defined as
\begin{equation}
G(x,y)=\sum_{t=1}^T\sum_{s=1,\ldots,T,s\neq t}G_{ts}(x,y) \label{G definition}
\end{equation}
where
\begin{equation}
G_{ts}(x,y)=E_0[h(\mathbf{X}_T)|X_t=x,X_s=y]. \label{Gts}
\end{equation}
Also, $X$ and $Y$ are independent random variables each having distribution $P_0$. \label{thm:fixed time}
\end{thm}

The proof of Theorem \ref{thm:fixed time} is laid out in Section \ref{sec:main}.

\subsection{Extension to Random Time Horizon Problems} \label{section:stopping time}
Theorem \ref{thm:fixed time} can be generalized to some extent to problems involving a random time horizon $\tau$. Consider the cost function $h(\mathbf{X}_\tau)$ that depends on the sequence $\mathbf{X}_\tau=(X_1,X_2,\ldots,X_\tau)$. Formulation \eqref{max fixed time} is replaced by
\begin{equation}
\begin{array}{ll}
\max&E_f[h(\mathbf{X}_\tau)]\\
\text{subject to}\ &D(P_f\|P_0)\leq\eta\\
&X_t\stackrel{i.i.d.}{\sim}P_f\text{\ \ for\ }t=1,2,\ldots\\
&P_f\in\mathcal{P}_0
\end{array} \label{max stopping time}
\end{equation}
where $E_f[\cdot]$ is the corresponding expectation with respect to $\mathbf X_\tau$.

To state the result in this direction, we impose either a boundedness or an independence condition on $\tau$:


\begin{assumption}
The random time $\tau$ is a stopping time with respect to $\{\mathcal F_t\}_{t\geq1}$, a filtration that supersets the filtration generated by the sequence $\{X_t\}_{t\geq1}$, namely $\{\mathcal{F}(X_1,\ldots,X_t)\}_{t\geq1}$. Moreover, $\tau$ is bounded a.s. by a deterministic time $T$. The cost function $h$ satisfies $|h(\mathbf{X}_\tau)|\leq\sum_{t=1}^T\Lambda_t(X_t)$ a.s. for some deterministic functions $\Lambda_t(\cdot)$, where $\Lambda_t(X_t)$ each possesses finite exponential moment, i.e.~$E_0[e^{\theta \Lambda_t(X)}]<\infty$, for $\theta$ in a neighborhood of zero. \label{bounded stopping time}
\end{assumption}

%

\begin{assumption}
The random time $\tau$ is independent of the sequence $\{X_t\}_{t\geq1}$, and has finite second moment under $P_0$, i.e.~$E_0\tau^2<\infty$. Moreover, the cost function $h(\mathbf{X}_\tau)$ is bounded a.s.. \label{independence}
\end{assumption}

Next, we also place a non-degeneracy condition analogous to Assumption \ref{not sup}. We define $\tilde{g}:\mathcal{X}\to\mathbb{R}$ as
\begin{equation}
\tilde{g}(x)=\sum_{t=1}^\infty\tilde{g}_t(x) \label{gtilde definition}
\end{equation}
where $\tilde{g}_t(x)$ is given by
\begin{equation}
\tilde{g}_t(x)=E_0[h(\mathbf{X}_\tau);\tau\geq t|X_t=x]. \label{gtilde t}
\end{equation}
Our non-degeneracy condition is now imposed on the function $\tilde{g}$ acted on $X\sim P_0$:
\begin{assumption}
The random variable $\tilde{g}(X)$ is non-constant under $P_0$.
\label{not sup stopping time}
\end{assumption}
 %
%
We have the following theorem:
\begin{thm}
With either Assumption \ref{bounded stopping time} or \ref{independence} in hold, together with Assumption \ref{not sup stopping time}, the optimal value of \eqref{max stopping time} satisfies
\begin{equation}
\max E_f[h(\mathbf{X}_\tau)]=E_0[h(\mathbf{X}_\tau)]+\sqrt{2Var_0(\tilde{g}(X))}\eta^{1/2}+\frac{1}{Var_0(\tilde{g}(X))}\left(\frac{1}{3}\kappa_3(\tilde{g}(X))+\tilde{\nu}\right)
\eta+O(\eta^{3/2}) \label{expansion stopping time}
\end{equation}
where
\begin{equation}
\tilde{\nu}=E_0[(\tilde{G}(X,Y)-E_0[\tilde{G}(X,Y)])(\tilde{g}(X)-E_0[\tilde{g}(X)])(\tilde{g}(Y)-E_0[\tilde{g}(Y)])]. \label{nu tilde}
\end{equation}
Here $\tilde{g}(x)$ is defined in \eqref{gtilde definition} and \eqref{gtilde t}, and $\tilde{G}(x,y)$ is defined as $\tilde{G}(x,y)=\sum_{t=1}^\infty\sum_{\substack{s\geq1\\s\neq t}}\tilde{G}_{ts}(x,y)$, where $\tilde{G}_{ts}(x,y)$ is given by
$$\tilde{G}_{ts}(x,y)=E_0[h(\mathbf{X}_\tau);\tau\geq t\wedge s|X_t=x,X_s=y].$$\label{thm:stopping time}
\end{thm}

When $\tau$ is a finite deterministic time, Theorem \ref{thm:stopping time} reduces to Theorem \ref{thm:fixed time}. Further relaxation of Assumptions \ref{bounded stopping time} or \ref{independence} to more general stopping times is out of the scope of the present work and will be left elsewhere.

%
\subsection{Discussions}\label{sec:discussions}
We close this section with some discussions:
\\

\noindent 1. Similar results to Theorems \ref{thm:basic}, \ref{thm:fixed time} and \ref{thm:stopping time} hold if maximization formulation is replaced by minimization. Under the same assumptions, the first order term in all the expansions above will have a sign change for minimization formulation, while the second order term will remain the same. For example, the expansion for Theorem \ref{thm:fixed time} becomes
    \begin{equation}
\min E_f[h(\mathbf{X}_T)]=E_0[h(\mathbf{X}_T)]-\sqrt{2Var_0(g(X))}\eta^{1/2}+\frac{1}{Var_0(g(X))}\left(\frac{1}{3}\kappa_3(g(X))+\nu\right)\eta+O(\eta^{3/2}). \end{equation}
For Theorem \ref{thm:basic}, the change in \eqref{full relation} for the minimization formulation is that $\beta^*$ becomes the unique negative solution of the same equation.

These changes can be seen easily by merely replacing $h$ by $-h$ in the analysis.
\\

\noindent 2. The function $g(\cdot)$ defined in \eqref{g definition} is the Gateaux derivative of $E_0[h(\mathbf X_T)]$ with respect to the distribution $P_0$, viewing $E_0[h(\mathbf X_T)]$ as a functional of $P_0$. To illustrate what we mean, consider a perturbation of the probability distribution from $P_0$ to a mixture distribution $(1-\epsilon)P_0+\epsilon Q$ where $Q$ is a probability measure on $\mathcal X$ and $0<\epsilon<1$. Under suitable integrability conditions, one can check that
\begin{equation}
\frac{d}{d\epsilon}\int h(x_1,\ldots,x_T)\prod_{t=1}^Td((1-\epsilon)P_0(x_t)+\epsilon Q(x_t))\Bigg|_{\epsilon=0}=\int g(x)d(Q(x)-P_0(x)).\label{Gateaux}
\end{equation}
In the statistics literature, the function $g(x)-E_0[g(X)]$ has been known as the influence function \cite{hampel1974influence} in which $X_1,\ldots,X_T$ would play the role of i.i.d.~data. Influence functions have been used in measuring the effect on given statistics due to outliers or other forms of data contamination \cite{hampel1974influence,hampel2011robust}.
\\

\noindent 3. Our asymptotic expansions suggest that the square root of KL divergence is the correct scaling of the first order model misspecification effect. We will also show, in Section \ref{sec:others}, that our first order expansion coefficients dominate any first order parametric derivatives under a suitable rescaling from Euclidean distance to KL divergence.
\\

%
%


\noindent 4. Our results can be generalized to situations with multiple random sources and when one is interested in evaluating the model misspecification effect from one particular source. To illustrate, consider $E[h(\mathbf{X}_T,\mathbf{Y})]$ where $\mathbf{Y}$ is some random object potentially dependent of the i.i.d.~sequence $\mathbf X_T=(X_t)_{t=1,\ldots,T}$. Suppose the model for $Y$ is known and the focus is on assessing the effect of model misspecification for $X_t$. Theorem \ref{thm:fixed time} still holds with $h(\mathbf{X}_T)$ replaced by $E[h(\mathbf{X}_T,\mathbf{Y})|\mathbf{X}_T]$, where $E[\cdot]$ is with respect to the known distribution of $\mathbf Y$. This modification can be seen easily by considering $E_f[E[h(\mathbf{X}_T,\mathbf{Y})|\mathbf{X}_T]]$ as the performance measure and $E[h(\mathbf{X}_T,\mathbf{Y})|\mathbf{X}_T]$ as the cost function. Analogous observations apply to Theorems \ref{thm:basic} and \ref{thm:stopping time}.
\\

\noindent 5. KL divergence is a natural choice of statistical distance, as it has been used in model selection in statistics (e.g.~in defining Akaike Information Criterion \cite{akaike1974new}), possesses information theoretic properties \cite{kl51,kullback67,kemperman69,csiszar67}, and is transformation invariant \cite{cover2012elements}. Nevertheless, there are other possible choices of statistical distances, such as those in the $\phi$-divergence class \cite{pardo2005statistical}. 

\section{Connections to Past Literatures} \label{sec:literature}
Here we briefly review two lines of past literatures that are related to our work. First, the worst-case optimization and the use of statistical distance that we consider is related to robust control \cite{hs08} and distributionally robust optimization \cite{ben2013robust,delage2010distributionally}. These literatures consider decision making when full probabilistic description of the underlying model is not available. The problems are commonly set in terms of a minimax objective, where the maximum is taken over a class of models that is believed to contain the truth, often called the uncertainty set \cite{goh2010distributionally,lss06,bb12}. The use of statistical distance such as KL divergence in defining uncertainty set is particularly popular for dynamic control problems \cite{nilim2005robust,iyengar2005robust,pjd00}, economics \cite{hs01,hsAER01,hs03}, finance \cite{chsw02,cwz05,gx12a}, queueing \cite{jls10}, and dynamic pricing \cite{ls07}. In particular, \cite{gx12b} proposes the use of simulation, which they called \emph{robust Monte Carlo}, in order to approximate the solutions for a class of worst-case optimizations that arise in finance. Nevertheless, in all the above literatures, the typical focus is on the tractability of optimization formulations, which often include convexity. Instead, this paper provides a different line of analysis using asymptotic approximations for formulations that are intractable via developed methods yet arise naturally in stochastic modeling.



The second line of related literatures is sensitivity analysis. The surveys \cite{lecuyer91}, \cite{fu06}, \S VII in \cite{ag} and \S7 in \cite{glasserman} provide general overview on different methods for derivative estimation in classical sensitivity analysis, which focus on parametric uncertainty. 
Another notable area is perturbation analysis of Markov chains. These results are often cast as Taylor series expansions in terms of the perturbation of the transition matrix (e.g.~\cite{schweitzer1968perturbation,cao1998maclaurin,heidergott2010series}), where the distances defining the perturbations are typically matrix norms on the transition kernels rather than statistical distances defined between distributions. 
We also note the area of variance-based global sensitivity analysis \cite{saltelli2008global}. This often involves the estimation of the variance of conditional expectations on some underlying parameters, which resembles to some extent the form of the first order coefficient in our main theorems. 
The randomness in this framework can be interpreted from a Bayesian \cite{saltelli2010variance,oakley2004probabilistic} or data-driven \cite{zouaoui2003accounting,ankenman2012quick,barton2013quantifying} perspective, or in the context of Bayesian model averaging, the posterior variability among several models \cite{zouaoui2004accounting}. All these are nonetheless parametric-based.

\section{Mathematical Developments for Finite Horizon Problems} \label{sec:main}
In this section we lay out the analysis of the worst-case optimization for finite time horizon problems when $T>1$, leading to Theorem \ref{thm:fixed time}. Leveraging the idea in Section \ref{section:basic}, we first write the maximization problem \eqref{max fixed time} in terms of likelihood ratio $L$:
\begin{equation}
\begin{array}{ll}
\max&E_0[h(\mathbf{X}_T)\underline{L}_T]\\
\text{subject to}\ &E_0[L(X)\log L(X)]\leq\eta\\
&L\in\mathcal{L}
\end{array} \label{max L}
\end{equation}
where for convenience we denote $\underline{L}_T=\prod_{t=1}^TL(X_t)$, and $X$ as a generic variable that is independent of $\{X_t\}_{t=1,\ldots,T}$ and having identical distribution as each $X_t$. We will follow the recipe from Section \ref{section:basic} to prove Theorem \ref{thm:fixed time}:

\begin{enumerate}
\item Consider the Lagrangian relaxation of \eqref{max L}, and characterize its optimal solution.

\item Find the optimality conditions for \eqref{max L} in terms of the Lagrange multiplier and the Lagrangian relaxation.

\item Using these conditions, expand the optimal value of \eqref{max L} in terms of the Lagrange multiplier and subsequently $\eta$.
\end{enumerate}

The main technical challenge on implementing the above scheme is the product form $\underline{L}_T$ that appears in the objective function in \eqref{max L}. In this regard, our key development is a characterization of the optimal solution of the Lagrangian relaxation via a fixed point equation on a suitable functional space. 

For technical reason, we will look at an equivalent problem with a modified space of $L$ and will introduce a suitable norm and metric. Let $\Lambda(x)=\sum_{t=1}^T\Lambda_t(x)$ where $\Lambda_t(x)$ is defined in Assumption \ref{boundedness}. Define
$$\mathcal{L}(M)=\{L\in\mathcal{L}:E_0[\Lambda(X)L(X)]\leq M\}$$
for $M>0$, and the associated norm $\|L\|_\Lambda:=E_0[(1+\Lambda(X))L(X)]$ and metric
\begin{equation}
\|L-L'\|_\Lambda=E_0[(1+\Lambda(X))|L(X)-L'(X)|]. \label{d metric}
\end{equation}
It is routine to check that $\mathcal{L}(M)$ is complete. We have the following observation:
\begin{lemma}
For any $\eta\leq N$ for some small $N$, formulation \eqref{max L} is equivalent to
\begin{equation}
\begin{array}{ll}
\max&E_0[h(\mathbf{X}_T)\underline{L}_T]\\
\text{subject to}\ &E_0[L(X)\log L(X)]\leq\eta\\
&L\in\mathcal{L}(M)
\end{array} \label{max L M}
\end{equation}
for some large enough $M>0$, independent of $\eta$. \label{modification}
\end{lemma}

The proof of Lemma \ref{modification} is left to Section \ref{sec:proofs}. From now on we will focus on \eqref{max L M}. Its Lagrangian relaxation is given by
\begin{equation}
\max_{L\in\mathcal{L}(M)}E_0[h(\mathbf{X}_T)\underline{L}_T]-\alpha(E_0[L\log L]-\eta). \label{max dual}
\end{equation}
Our optimality characterization for \eqref{max dual} is:

\begin{proposition}
Under Assumption \ref{boundedness}, when $\alpha>0$ is large enough, the unique optimal solution of \eqref{max dual} satisfies
\begin{equation}
L(x)=\frac{e^{g^L(x)/\alpha}}{E_0[e^{g^L(X)/\alpha}]} \label{optimal L fixed time}
\end{equation}
where $g^L(x)=\sum_{t=1}^Tg_t^L(x)$ and $g_t^L(x)=E_0\left[h(\mathbf{X}_T)\prod_{\substack{1\leq r\leq T\\r\neq t}}L(X_r)\Big|X_t=x\right]$. \label{prop:fixed point}
\end{proposition}

The form in \eqref{optimal L fixed time} can be guessed from a heuristic differentiation with respect to $L$. To see this, consider further relaxation of the constraint $E_0[L]=1$ in \eqref{max dual}:
\begin{equation}
E_0[h(\mathbf{X}_T)L(X_1)L(X_2)\cdots L(X_T)]-\alpha E_0[L\log L]+\alpha\eta+\lambda E_0[L]-\lambda. \label{interim fixed time}
\end{equation}
There are $T$ factors of $L$ in the first term. A heuristic ``product rule" of differentation is to sum up the derivative with respect to each $L$ factor, keeping all other $L$'s unchanged. To do so, we condition on $X_t$ to write
$$E_0[h(\mathbf{X}_T)L(X_1)L(X_2)\cdots L(X_T)]=E_0\left[E_0\left[h(\mathbf{X}_T)\prod_{\substack{1\leq r\leq T\\r\neq t}}L(X_r)\Bigg|X_t\right]L(X_t)\right]$$
and
\begin{equation}
\frac{d}{dL(x)}E_0\left[h(\mathbf{X}_T)\prod_{t=1}^TL(X_t)\right]\ ``="\ \sum_{t=1}^TE_0\left[h(\mathbf{X}_T)\prod_{\substack{1\leq r\leq T\\r\neq t}}L(X_r)\Bigg|X_t=x\right]. \label{heuristic differentiation}
\end{equation}
So the Euler-Lagrange equation is
\begin{equation}\sum_{t=1}^TE_0\left[h(\mathbf{X}_T)\prod_{\substack{1\leq r\leq T\\r\neq t}}L(X_r)\Bigg|X_t=x\right]-\alpha\log L(x)-\alpha+\lambda=0 \label{Euler Lagrange}
\end{equation}
which gives
$$L(x)\propto\exp\left\{\sum_{t=1}^TE_0\left[h(\mathbf{X}_T)\prod_{\substack{1\leq r\leq T\\r\neq t}}L(X_r)\Bigg|X_t=x\right]\Bigg/\alpha\right\}.$$
The constraint $E_0[L]=1$ then gives the expression \eqref{optimal L fixed time}. The ``product rule" \eqref{heuristic differentiation} can be readily checked for finitely supported $X$. The following shows an instance when $T=2$:
\begin{example}
Consider two i.i.d.~random variables $X_1$ and $X_2$, and a cost function $h(X_1,X_2)$. The variables $X_1$ and $X_2$ have finite support on $1,2,\ldots,n$ under $P_0$. Denote $p(x)=P_0(X_1=x)$ for $x=1,2,\ldots,n$. The objective value in \eqref{max L} in this case is
$$E_0[h(X_1,X_2)L(X_1)L(X_2)]=\sum_{x_1=1}^n\sum_{x_2=1}^nh(x_1,x_2)p(x_1)p(x_2)L(x_1)L(x_2).$$
Now differentiate with respect to each $L(1),L(2),\ldots,L(n)$ respectively. For $i=1,\ldots,n$, we have
\begin{align*}
\frac{d}{dL(i)}E_0[h(X_1,X_2)]&=\sum_{x_1\neq i}h(x_1,i)p(x_1)p(i)L(x_1)+\sum_{x_2\neq i}h(i,x_2)p(i)p(x_2)L(x_2)+2h(i,i)p(i)^2L(i)\\
&=E_0[h(X_1,i)L(X_1)]+E_0[h(i,X_2)L(X_2)].
\end{align*}
This coincides with the product rule \eqref{heuristic differentiation} discussed above.
\end{example}

\subsection{Outline of Argument of Proposition \ref{prop:fixed point}} \label{T general}


The proof of Proposition \ref{prop:fixed point} centers around an operator $\mathcal{K}:\mathcal{L}(M)^{T-1}\to\mathcal{L}(M)^{T-1}$ as follows. First, we define a function derived from $h$ as
\begin{equation}
S_h(\mathbf{x}_T)=\sum_{\mathbf{y}\in\mathcal{S}_T}h(\mathbf{y}) \label{bar S}
\end{equation}
where $\mathcal{S}_T$ is the symmetric group of all permutations of $\mathbf{x}_T=(x_1,\ldots,x_T)$. The summation in \eqref{bar S} has $T!$ number of terms. Obviously, by construction the value of $S_h$ is invariant to any permutation of its arguments.

Denote $\mathbf{L}=(L_1,\ldots,L_{T-1})\in\mathcal{L}(M)^{T-1}$. We now define a mapping $K:\mathcal{L}(M)^{T-1}\to\mathcal{L}(M)$ given by
\begin{equation}
K(L_1,\ldots,L_{T-1})(x):=\frac{e^{E_0[S_h(X,X_1,\ldots,X_{T-1})\prod_{t=1}^{T-1}L_t(X_t)|X=x]/(\alpha(T-1)!)}}{E_0[e^{E_0[S_h(X,X_1,\ldots,X_{T-1})\prod_{t=1}^{T-1}L_t(X_t)|X]/(\alpha(T-1)!)}]} \label{component K}
\end{equation}
where $X,X_1,X_2,\ldots,X_{T-1}$ are i.i.d.~random variables with distribution $P_0$.
Then for a given $\mathbf{L}$, define
\begin{equation}
\begin{array}{ll}
\tilde{L}_1&=K(L_1,\ldots,L_{T-1})\\
\tilde{L}_2&=K(\tilde{L}_1,L_2,\ldots,L_{T-1})\\
\tilde{L}_3&=K(\tilde{L}_1,\tilde{L}_2,L_3,\ldots,L_{T-1})\\
&\vdots\\
\tilde{L}_{T-1}&=K(\tilde{L}_1,\ldots,\tilde{L}_{T-2},L_{T-1}).
\end{array} \label{relation K}
\end{equation}
Finally, the operator $\mathcal{K}$ on $\mathcal{L}(M)^{T-1}$ is defined as
\begin{equation}
\mathcal{K}(\mathbf{L})=(\tilde{L}_1,\ldots,\tilde{L}_{T-1}). \label{multivariate K}
\end{equation}

%

The following shows the main steps for the proof of Proposition \ref{prop:fixed point}.
\\

\noindent\textbf{Step 1: Contraction Mapping. }We have:
\begin{lemma}
Under Assumption \ref{boundedness}, when $\alpha$ is sufficiently large, the operator $\mathcal{K}:\mathcal{L}(M)^{T-1}\to\mathcal{L}(M)^{T-1}$ defined in \eqref{multivariate K} is well-defined, closed and a contraction, using the metric $d(\cdot,\cdot):\mathcal{L}(M)^{T-1}\times\mathcal{L}(M)^{T-1}\to\mathbb{R}_+$ defined as
$$d(\mathbf{L},\mathbf{L}')=\max_{t=1,\ldots,T}E_0[(1+\Lambda(X))|L_t(X)-L_t'(X)|]$$
where $\mathbf{L}=(L_1,\ldots,L_{T-1}),\mathbf{L}'=(L_1',\ldots,L_{T-1}')\in\mathcal{L}(M)^{T-1}$.
Hence there exists a unique fixed point $\mathbf{L}^*$ that satisfies $\mathcal{K}(\mathbf{L}^*)=\mathbf{L}^*$. Moreover, all components of $\mathbf{L}^*$ are identical. \label{lemma:multivariate fixed point}
\end{lemma}

This leads to a convergence result on the iteration driven by the mapping $K$:
\begin{corollary}
With Assumption \ref{boundedness} and sufficiently large $\alpha$, starting from any $L^{(1)},\ldots,L^{(T-1)}\in\mathcal{L}(M)$, the iteration $L^{(k)}=K(L^{(k-T+1)},\ldots,L^{(k-1)})$ for $k\geq T$, where $K:\mathcal{L}(M)^{T-1}\to\mathcal{L}(M)$ is defined in \eqref{component K}, converges to $L^*$ in $\|\cdot\|_\Lambda$-norm, where $L^*$ is the identical component of the fixed point $\mathbf{L}^*$ of $\mathcal{K}$. Moreover, $L^*=K(L^*,\ldots,L^*)$. \label{corollary:component}
\end{corollary}

\noindent\textbf{Step 2: Monotonicity of the Objective Value under Iteration of $K$. }We shall consider the objective in \eqref{max dual} multiplied by $T!$, i.e.
\begin{equation}
T!(E_0[h(\mathbf{X}_T)\underline{L}_T]-\alpha E_0[L\log L])=E_0[S_h(\mathbf{X}_T)\underline{L}_T]-\alpha T!E_0[L\log L]. \label{scaled objective}
\end{equation}
Iterations driven by the mapping $K$ possess a monotonicity property on this scaled objective:
\begin{lemma}
With Assumption \ref{boundedness} and sufficiently large $\alpha$, starting from any $L^{(1)},\ldots,L^{(T)}\in\mathcal{L}(M)$, construct the sequence $L^{(k+1)}=K(L^{(k-T+2)},\ldots,L^{(k)})$ for $k\geq T$, where $K$ is defined in \eqref{component K}. Then
\begin{equation}
E_0\left[S_h(\mathbf{X}_T)\prod_{t=1}^TL^{(k+t-1)}(X_t)\right]-\alpha(T-1)!\sum_{t=1}^TE_0[L^{(k+t-1)}\log L^{(k+t-1)}] \label{convergence multivariate}
\end{equation} is non-decreasing in $k$, for $k\geq1$. \label{lemma:multivariate monotonicity}
\end{lemma}

\noindent\textbf{Step 3: Convergence of the Objective Value to the Optimum. }Finally, we have the convergence of \eqref{convergence multivariate} to the scaled objective \eqref{scaled objective} evaluated at any identical component of the fixed point of $\mathcal{K}$:
\begin{lemma}
With Assumption \ref{boundedness} and sufficiently large $\alpha$, starting from any $L^{(1)},\ldots,L^{(T)}\in\mathcal{L}(M)$, we have
\begin{eqnarray*}
&&E_0\left[S_h(\mathbf{X}_T)\prod_{t=1}^TL^{(k+t-1)}(X_t)\right]-\alpha(T-1)!\sum_{t=1}^TE_0[L^{(k+t-1)}\log L^{(k+t-1)}]\\
&\to&E_0\left[S_h(\mathbf{X}_T)\prod_{t=1}^TL^*(X_t)\right]-\alpha T!E_0[L^*\log L^*]
\end{eqnarray*}
where $L^{(k)}$ is defined by the same recursion in Lemma \ref{lemma:multivariate monotonicity}, and $L^*$ is any identical component of $\mathbf{L}^*\in\mathcal{L}(M)^{T-1}$, the fixed point of $\mathcal{K}$ defined in \eqref{multivariate K}. \label{lemma:multivariate convergence}
\end{lemma}

These lemmas will be proved in Section \ref{sec:proof fixed point}. Once they are established, Proposition \ref{prop:fixed point} follows immediately:
\begin{proofof}{Proposition \ref{prop:fixed point}}
Given any $L\in\mathcal{L}(M)$, Lemmas \ref{lemma:multivariate monotonicity} and \ref{lemma:multivariate convergence} together conclude that
$$E_0\left[S_h(\mathbf{X}_T)\prod_{t=1}^TL(X_t)\right]-\alpha(T-1)!\sum_{t=1}^TE_0[L\log L] \leq E_0\left[S_h(\mathbf{X}_T)\prod_{t=1}^TL^*(X_t)\right]-\alpha T!E_0[L^*\log L^*]$$
by defining $L^{(1)}=\cdots=L^{(T)}=L$ and using the recursion defined in the lemmas. Here $L^*$ is the identical component of the fixed point of $\mathcal{K}$. By Corollary \ref{corollary:component}, $L^*=K(L^*,\ldots,L^*)$ so $L^*$ satisfies \eqref{optimal L fixed time}. This concludes Proposition \ref{prop:fixed point}.
\end{proofof}

\subsection{Asymptotic Expansions}
The characterization of $L^*$ in Proposition \ref{prop:fixed point} can be used to obtain asymptoptic expansion of $L^*$ in terms of $\alpha^*$. The proof of Theorem \ref{thm:fixed time}, as outlined in the recipe at the beginning of this section, then follows from an elaboration of the machinery developed in Section \ref{section:basic}. 
Details are provided in Section \ref{sec:proof asymptotics}.

\section{Extension to Random Time Horizon Problems} \label{sec:extension random time}

We shall discuss here the extension to random time horizon problems under Assumption \ref{bounded stopping time}, using the result in Theorem \ref{thm:fixed time}: 

\begin{proofof}{Theorem \ref{thm:stopping time} under Assumption \ref{bounded stopping time}}
First, formulation \eqref{max stopping time} can be written in terms of likelihood ratio:
\begin{equation}
\begin{array}{ll}
\max&E_0[h(\mathbf{X}_\tau)\underline{L}_\tau]\\
\text{subject to}\ &E_0[L(X)\log L(X)]\leq\eta\\
&L\in\mathcal{L}
\end{array} \label{max stopping time L}
\end{equation}
where $\underline{L}_\tau=\prod_{t=1}^\tau L(X_t)$. 
Under Assumption \ref{bounded stopping time}, $\tau\leq T$ a.s., for some $T>0$. Hence the objective in \eqref{max stopping time L} can also be written as $E_0[h(\mathbf{X}_\tau)\underline{L}_\tau]=E_0[h(\mathbf{X}_\tau)\underline{L}_T]$ by the martingale property of $\underline{L}_t$.

This immediately falls back into the framework of Theorem \ref{thm:fixed time}, with the cost function now being $h(\mathbf{X}_\tau)$. For this particular cost function, we argue that the $g(x)$ and $G(x,y)$ in Theorem \ref{thm:fixed time} are indeed in the form stated in Theorem \ref{thm:stopping time}. To this end, we write
\begin{equation}
g(x)=\sum_{t=1}^TE_0[h(\mathbf{X}_\tau)|X_t=x]=\sum_{t=1}^TE_0[h(\mathbf{X}_\tau);\tau\geq t|X_t=x]+\sum_{t=1}^TE_0[h(\mathbf{X}_\tau);\tau<t|X_t=x]. \label{reduction stopping}
\end{equation}
Consider the second summation in \eqref{reduction stopping}. Since $h(\mathbf{X}_\tau)I(\tau<t)$ is $\mathcal{F}_{t-1}$-measurable, it is independent of $X_t$. As a result, the second summation in \eqref{reduction stopping} is constant. Similarly, we can write
\begin{align}
G(x,y)&=\sum_{t=1}^T\sum_{\substack{s=1,\ldots,T\\s\neq t}}E_0[h(\mathbf{X}_\tau)|X_t=x,X_s=y] \notag\\
&=\sum_{t=1}^T\sum_{\substack{s=1,\ldots,T\\s\neq t}}E_0[h(\mathbf{X}_\tau);\tau\geq t\wedge s|X_t=x,X_s=y]+\sum_{t=1}^T\sum_{\substack{s=1,\ldots,T\\s\neq t}}E_0[h(\mathbf{X}_\tau);\tau<t\wedge s|X_t=x,X_s=y] \label{reduction stopping 1}
\end{align}
and the second summation in \eqref{reduction stopping 1} is again a constant. It is easy to check that the first and second order coefficients in Theorem \ref{thm:fixed time} are translation invariant to $g(x)$ and $G(x,y)$ respectively, i.e.~adding a constant in $g(x)$ or $G(x,y)$ does not affect the coefficients. Therefore Theorem \ref{thm:stopping time} follows immediately.
\end{proofof}

The proof of Theorem \ref{thm:stopping time} under Assumption \ref{independence} builds on the above argument by considering a sequence of truncated random time $\tau\wedge T,\ T=1,2,\ldots$. 
We defer its details to Section \ref{sec:proof stopping time}.

\section{Bounds on Parametric Derivatives} \label{sec:others}
The coefficients in our expansions in Section \ref{sec:results} dominate any parametric derivatives in the following sense:
\begin{proposition}
Suppose $P_0$ lies in a parametric family $P^\theta$ with $\theta\in\Theta\subset\mathbb{R}$, say $P_0=P^{\theta_0}$ where $\theta_0\in\Theta^\circ$. Denote $E^\theta[\cdot]$ as the expectation under $P^\theta$. Assume that
\begin{enumerate}
\item $P^\theta$ is absolutely continuous with respect to $P^{\theta_0}$ for $\theta$ in a neighborhood of $\theta_0$.
\item $D(\theta,\theta_0):=D(P^\theta\|P^{\theta_0})\to0$ as $\theta\to\theta_0$.
\item For any $\eta$ in a neighborhood of 0 (but not equal to 0), $D(\theta,\theta_0)=\eta$ has two solutions $\theta^+(\eta)>\theta_0$ and $\theta^-(\eta)<\theta_0$; moreover,
    $$\frac{d}{d\theta}\sqrt{D(\theta,\theta_0)}\Big|_{\theta=\theta_0^+}>0$$
    and
    $$\frac{d}{d\theta}\sqrt{D(\theta,\theta_0)}\Big|_{\theta=\theta_0^-}<0.$$
\item $\frac{d}{d\theta}E^\theta[h(\mathbf{X})]\Big|_{\theta=\theta_0}$ exists.
\end{enumerate}
Then
$$\left|\frac{d}{d\theta}E^\theta[h(\mathbf{X})]\Big/\frac{d}{d\theta}\sqrt{D(\theta,\theta_0)}\Big|_{\theta=\theta_0^{\pm}}\right|\leq\sqrt{2Var_0(\zeta(X))}$$ where $\zeta$ is the function $h$, $g$ or $\tilde{g}$ in Theorems \ref{thm:basic}, \ref{thm:fixed time} and \ref{thm:stopping time} respectively, depending on the structure of $\mathbf{X}$ that is stated in each theorem under the corresponding assumptions. \label{prop:parametric dominance}
\end{proposition}

This proposition states the natural property that the first order expansion coefficients of worst-case optimizations dominate the parametric derivative taken in the more restrictive parametric model space. The proof is merely a simple application of the first principle of differentiation:
\begin{proof}
We consider only the setting in Theorem \ref{thm:fixed time}, as the others are similar. Let $P_0=P^{\theta_0}$. Denote $E_{f^+(\eta)}[h(\mathbf{X}_T)]$ as the optimal value of \eqref{max fixed time} and $E_{f^-(\eta)}[h(\mathbf{X}_T)]$ as the optimal value of \eqref{min fixed time}, when $\eta$ is in a neighborhood of 0. Under our assumptions, $P^{\theta^\pm(\eta)}$ with $D(\theta^\pm(\eta),\theta_0)=\eta$ are feasible solutions to both programs \eqref{max fixed time} and \eqref{min fixed time}, and hence the quantity $E^{\theta^\pm(\eta)}[h(\mathbf{X})]$ satisfies $E^{\theta^\pm(\eta)}[h(\mathbf{X}_T)]\leq E_{f^+(\eta)}[h(\mathbf{X}_T)]$ and $E^{\theta^\pm(\eta)}[h(\mathbf{X}_T)]\geq E_{f^-(\eta)}[h(\mathbf{X}_T)]$. This implies
$$\frac{E_{f^-(\eta)}[h(\mathbf{X}_T)]-E_0[h(\mathbf{X}_T)]}{\sqrt{\eta}}\leq\frac{E^{\theta^\pm(\eta)}[h(\mathbf{X}_T)]-E_0[h(\mathbf{X}_T)]}{\sqrt{\eta}}\leq\frac{E_{f^+(\eta)}[h(\mathbf{X}_T)]-E_0[h(\mathbf{X}_T)]}{\sqrt{\eta}}.$$
Taking the limit as $\sqrt{\eta}\to0$, the upper and lower bounds converge to $\sqrt{2Var_0(g(X))}$ by Theorem \ref{thm:fixed time} (and discussion point 1 in Section \ref{sec:discussions}). Moreover, the quantities
$$\lim_{\sqrt{\eta}\to0}\frac{E^{\theta^\pm(\eta)}[h(\mathbf{X}_T)]-E_0[h(\mathbf{X}_T)]}{\sqrt{\eta}}=\frac{d}{d\sqrt{\eta}}E^{\theta^\pm(\eta)}[h(\mathbf{X}_T)]\Big|_{\sqrt{\eta}=0}$$
become $\frac{d}{d\theta}E^\theta[h(\mathbf{X}_T)]\Big/\frac{d}{d\theta}\sqrt{D(\theta,\theta_0)}\Big|_{\theta=\theta_0^+}$ and $\frac{d}{d\theta}E^\theta[h(\mathbf{X}_T)]\Big/\frac{d}{d\theta}\sqrt{D(\theta,\theta_0)}\Big|_{\theta=\theta_0^-}$ respectively, by using chain rule and implicit function theorem. This concludes the proposition.
\end{proof}

\section{Numerical Examples} \label{sec:numerics}
We demonstrate some numerics of our results, in particular Theorem \ref{thm:fixed time}, using an example of multi-server queue. Consider a first-come-first-serve Markovian queue with $s$ number of servers. Customers arrive according to a Poisson process with rate $0.7s$ and enact i.i.d.~exponential service times with rate 1. Whenever the service capacity is full, newly arriving customers have to wait. We assume the system is initially empty. 
Our focus is to assess the effect if service times deviate from the exponential assumption. More concretely, let us consider our performance measure as the tail probability of the waiting time for the 100-th customer larger than a threshold 1.

To quantify the sensitivity of the exponential assumption for the service times, we compute the first order coefficient $\sqrt{2Var_0(g(X))}$ in Theorem \ref{thm:fixed time}, where $P_0$ is $Exp(1)$ and $g(\cdot)$ is computed by sequentially conditioning on the service time of customers 1 through 100, as defined in \eqref{g definition}. We tabulate, for $s=1,\ldots,5$, the point and interval estimates of the baseline performance measures and the first order coefficients $\sqrt{2Var_0(g(X))}$ in Table \ref{mms table}. Moreover, for each $s$, we calculate the ratio between the first order coefficient and the baseline performance measure as an indicator of the relative impact of model misspecification:
$$\text{Relative model misspecification impact}:=\frac{\text{Magnitude of first order coefficient}}{\text{Performance measure}}$$

\begin{table}[ht]
\centering
\begin{tabular}{c|cc|cc|c}
&\multicolumn{2}{c|}{Baseline performance measure}&\multicolumn{2}{c|}{First order coefficient}&Relative impact\\
Number of servers & Mean & $95\%$ C.I. & Mean & $95\%$ C.I.\\
\hline
1 & $0.519$ & $(0.518,0.520)$ & $1.685$ & $(1.566,1.805)$ & $3.248$\\
2 & $0.316$ & $(0.315,0.316)$& $1.689$ & $(1.556,1.822)$ &$5.353$ \\
3 & $0.200$ & $(0.199,0.201)$ & $1.446$ & $(1.318,1.573)$ & $7.239$ \\
4 & $0.129$ & $(0.128,0.129)$ & $1.217$ & $(1.079,1.355)$ & $9.460$ \\
5 & $0.084$ & $(0.083,0.084)$ & $0.957$ & $(0.856,1.058)$ & $11.462$
\end{tabular}
\caption{Simulation results for the performance measures and the first order coefficients in Theorem \ref{thm:fixed time} for the tail probability of waiting time of the 100-th customer in $M/M/s$ systems with different server capacities}
\label{mms table}
\end{table}

Table \ref{mms table} shows that the tail probability of the waiting time for the 100-th customer decreases from $0.52$ to $0.08$ as the number of servers increases from 1 to 5. The first order coefficient in Theorem \ref{thm:fixed time} also decreases in general from $1.69$ when $s=1$ to $0.96$ when $s=5$. The relative effect of model misspecification, on the other hand, increases from $3.25$ to $11.46$ as $s$ increases from 1 to 5. 

Figure \ref{mms figure} further depicts the first order approximations for the worst-case deviations $E_0[h(\mathbf{X}_T)]\pm\sqrt{2Var_0(g(X))\eta}$ for different levels of $\eta$ that represents the KL divergence. The solid line in the figure plots the baseline tail probability computed from using 1,000,000 samples for each $s=1,\ldots,5$. The dashed lines then show the approximate worst-case upper and lower bounds as $\eta$ increases. To get a sense on the magnitude of $\eta$, $\eta=0.005$ is equivalent to around $10\%$ discrepancy in service rate if the model is \emph{known} to lie in the family of exponential distribution; this can be seen by expressing the KL divergence in terms of service rate to see that roughly $KL\ divergence\approx(\%\ discrepancy\ in\ service\ rate)^2/2$ for small discrepancy. In fact, a service rate of $1.1$ corresponds to $\eta=0.0044$.

\begin{figure}[ht]
\centering
\includegraphics[scale=.23]{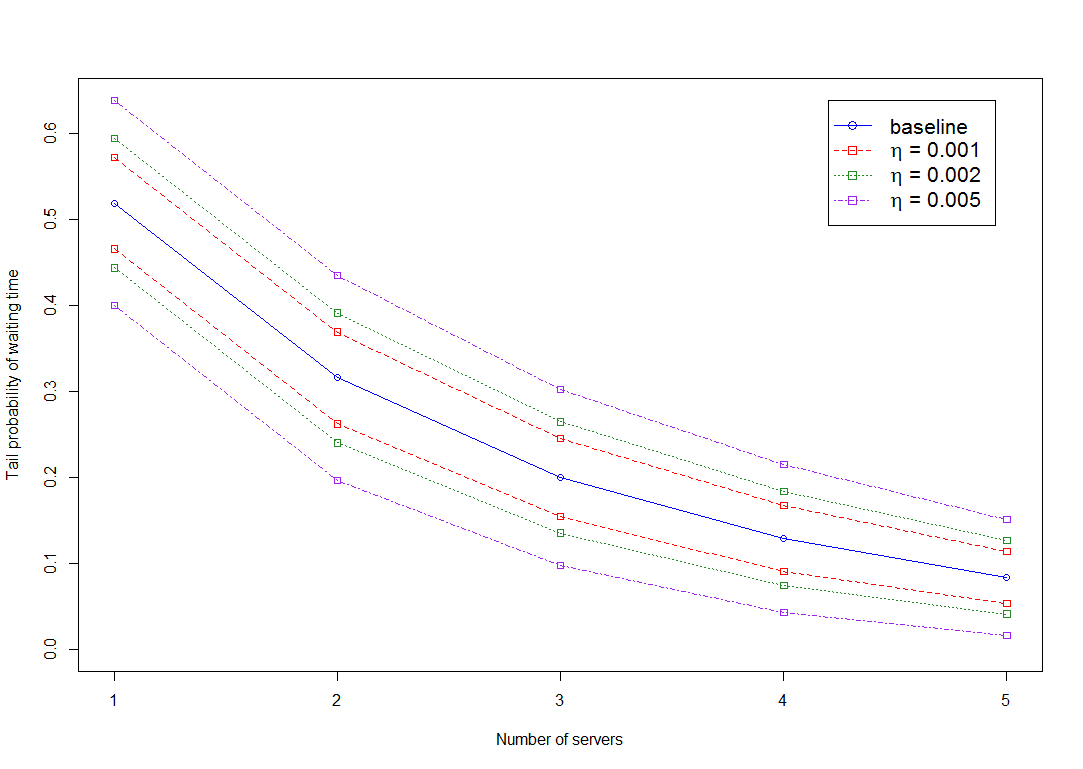}
\caption{Upper and lower first order approximations of the worst-case performance measures under different levels of input model discrepancies in terms of KL divergence for $M/M/1$ baseline system}
\label{mms figure}
\end{figure}

According to Proposition \ref{prop:fixed point}, the worst-case change of measure that gives rise to the values of the first order coefficients in Table \ref{mms table} satisfies $L(x)\propto e^{g^L(x)/\alpha^*}$, with $\alpha^*$ being the Lagrange multiplier, when $\eta$ is small. It is not possible to compute this change of measure exactly. What we can do, however, is to test our bounds from Theorem \ref{thm:fixed time} against some parametric models. Consider for instance $s=1$. The solid curves in Figure \ref{fig:comparison} plot the upper and lower bounds using only the first order approximations in Theorem \ref{thm:fixed time} (the surrounding dashed curves are the $95\%$ confidence bands for the bounds, and the dashed horizontal line is the baseline performance measure). For comparison, we simulate the performance measures using six different sequences of parametric models: the first two are kept as exponential distribution, with increasing and decreasing service rate starting from 1; the next two sequences are gamma distributions, one with the shape parameter varying from 1 and rate parameter kept at 1, whereas the other with rate parameter varying too so that the mean of the distribution is kept at 1; the last two sequences are mixtures of exponentials, one having two mixture components with rate 1 and 2 respectively and weight of the first component decreases from 1, whereas the other one having three components with the weights varying in a way that keeps the mean at 1.

As we can see, the first order bounds in Figure \ref{fig:comparison} appear to contain all the performance measures for $\eta$ up to $0.005$. It is expected that a second order correction would further improve the accuracy of the bounds. One side observation is that the sequences with the service times kept at mean equaling 1 are much closer to the baseline than the others.


\begin{figure}[ht]
\centering
\includegraphics[scale=.3]{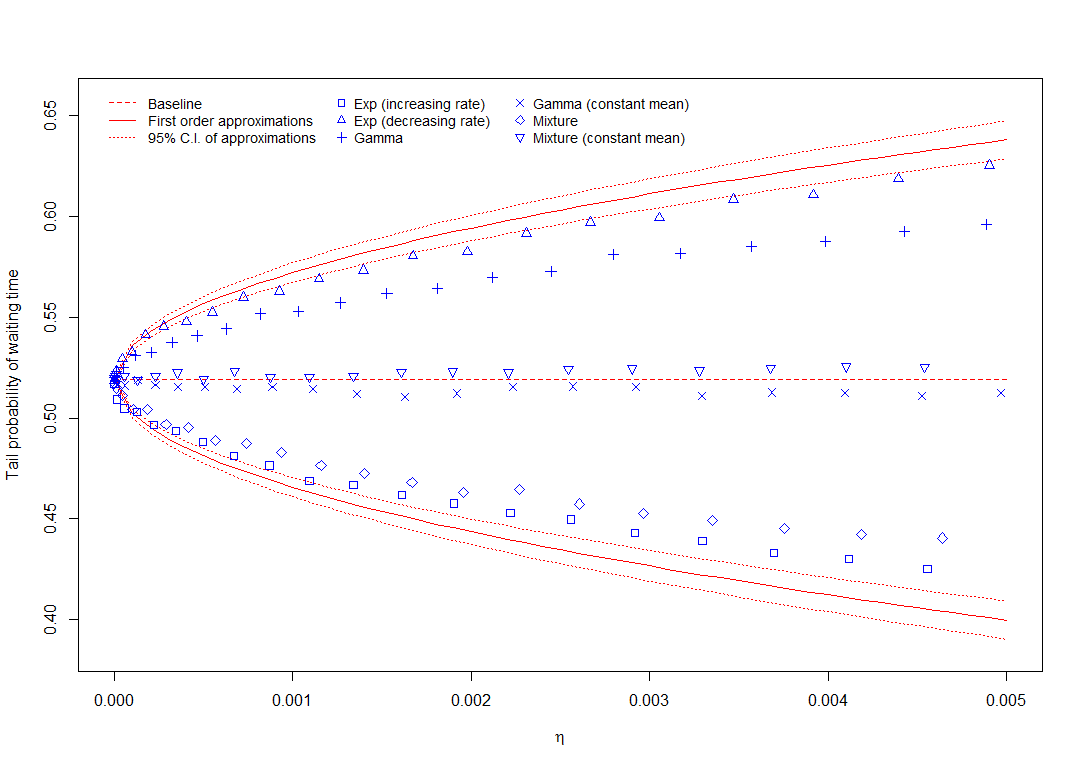}
\caption{Comparison of the first order approximations of the worst-case performance measures against parametric models}
\label{fig:comparison}
\end{figure}

The same methodology as above can be easily adapted to test other types of performance measures and models. For example, Table \ref{ggs table} and Figure \ref{ggs figure} carry out the same assessment scheme for the service time of a non-Markovian $G/G/s$ queue with gamma arrivals and uniform service times. Here we consider a deviation from the uniform distribution of the service time. In this scenario, we see from Table \ref{ggs table} that both the performance measures themselves and the magnitudes of first order coefficients are smaller than those in the $M/M/s$ case. Nonetheless the relative impacts are relatively similar.

In real applications, the magnitude of $\eta$ is chosen to represent the statistical uncertainty of the input model. Section 4.2 in \cite{gx12b} for instance provides some discussion on the choice based on past data. There are also studies on nonparametric estimation of KL divergence; see, for example, \cite{beirlant1997nonparametric} for a review of older works, \cite{paninski2003estimation}, and more recently \cite{liu2012exponential} and \cite{honorio2013two}.

\begin{figure}[ht]
\centering
\includegraphics[scale=.23]{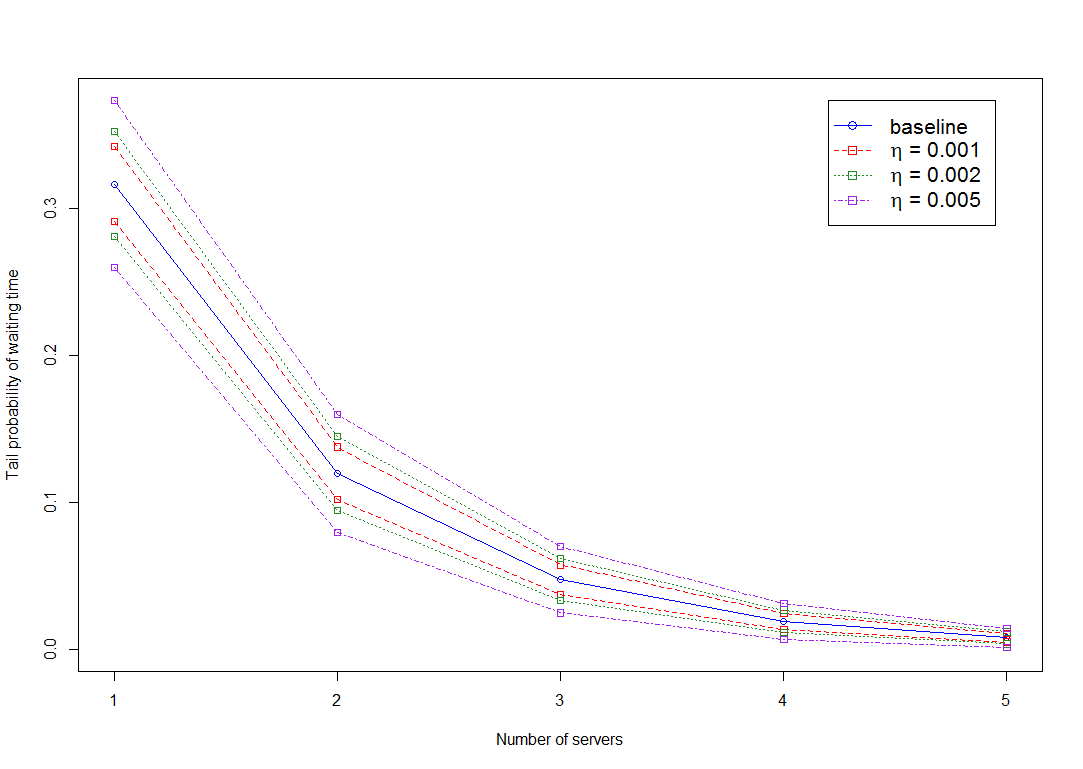}
\caption{Upper and lower first order approximations of the worst-case performance measures under different levels of input model discrepancies in terms of KL divergence for $G/G/1$ baseline system}
\label{ggs figure}
\end{figure}

\begin{table}[ht]
\centering
\begin{tabular}{c|cc|cc|c}
&\multicolumn{2}{c|}{Baseline performance measure}&\multicolumn{2}{c|}{First order coefficient}&Relative impact\\
Number of servers & Mean & C.I. & Mean & C.I.\\
\hline
$1$&$0.316$&$(0.315,0.317)$&$0.802$&$(0.737,0.867)$&$2.535$\\
$2$&$0.119$&$(0.119,0.120)$&$0.567$&$(0.531,0.602)$&$4.741$\\
$3$&$0.047$&$(0.047,0.048)$&$0.320$&$(0.298,0.341)$&$6.770$\\
$4$&$0.019$&$(0.019,0.019)$&$0.173$&$(0.160,0.185)$&$9.146$\\
$5$&$0.008$&$(0.008,0.008)$&$0.091$&$(0.079,0.104)$&$11.879$
\end{tabular}
\caption{Simulation results for the performance measures and the first order coefficients in Theorem \ref{thm:fixed time} for the tail probability of waiting time of the 100-th customer in $G/G/s$ systems with different server capacities}
\label{ggs table}
\end{table}

Finally, we explain in more detail our estimation procedure for $\sqrt{2Var_0(g(X))}$. Note first that our performance measure of interest depends on both the interarrival and service times, but the interarrival distribution is assumed known and so the cost function can be regarded as $E_0[h(\mathbf{X}_T,\mathbf{Y}_T)|\mathbf{X}_T]$, where $\mathbf{X}_T$ denotes the sequence of service times and $\mathbf{Y}_T$ as the interarrival times (see discussion point 4 in Section \ref{sec:discussions}). Second, note also that Assumption \ref{boundedness} is easily satisfied since $h$ as the indicator function is bounded. Moreover, Assumption \ref{not sup} is trivially verified by our computation demonstrating that $g(X)$ is not a constant. Hence the assumptions in Theorem \ref{thm:fixed time} are valid.

Now, it is easy to see that
$$g(x)=E_0\left[\sum_{t=1}^Th(\mathbf X_T^{(t)},\mathbf Y_T^{(t)})\Bigg|X=x\right]$$
where $\mathbf X_T^{(t)}=(X_1^{(t)},\ldots,X_{t-1}^{(t)},X,X_{t+1}^{(t)},\ldots,X_T^{(t)})$ and $\mathbf Y_T^{(t)}=(Y_1^{(t)},\ldots,Y_T^{(t)})$, with $X_s^{(t)}$ and $Y_s^{(t)}$ being i.i.d.~copies from the interarrival time and service time distributions respectively. Therefore $Var_0(g(X))$ is in the form of the variance of a conditional expectation, for which we can adopt an unbiased estimator from \cite{Sun2011}. This estimator takes the following form.
For convenience, denote $H:=\sum_{t=1}^Th(\mathbf X_T^{(t)},\mathbf Y_T^{(t)})$. To compute $Var_0(E[H|X])$, we carry out a nested simulation by first simulating $X_k,k=1,\ldots,K$, and then given each $X_k$, simulating $H_{kj},j=1,\ldots,n$. Then an unbiased estimator is
\begin{equation}
\widehat{\sigma_M^2}=\frac{1}{K-1}\sum_{k=1}^K(\bar{H}_k-\bar{\bar{H}})^2-\frac{1}{n}\widehat{{\sigma}_\epsilon^2} \label{sigma M}
\end{equation}
where
$$\widehat{{\sigma}_\epsilon^2}=\frac{1}{K(n-1)}\sum_{k=1}^K\sum_{j=1}^n(H_{kj}-\bar{H}_k)^2,\ \ \bar{H}_k=\frac{1}{n}\sum_{j=1}^nH_{kj}\text{\ \ and\ \ }\bar{\bar{H}}=\frac{1}{K}\sum_{k=1}^K\bar{H}_k.$$


To obtain a consistent point estimate and confidence interval for $\sqrt{2Var_0(g(X))}$, we use the delta method (see, for example, \S III in \cite{ag}). The overall sampling strategy is as follows:
\begin{enumerate}
\item Repeat the following $N$ times:
\begin{enumerate}
\item Simulate $K$ samples of $X$, say $X_k=x_k,k=1,\ldots,K$.
\item For each realized $x_k$, simulate $n$ samples of $H$ given $X=x_k$.
\item Calculate $\widehat{\sigma_M^2}$ using \eqref{sigma M}. 
    \end{enumerate}
\item The above procedure generates $N$ estimators $\widehat{\sigma_M^2}$. Call them $Z_l,l=1,\ldots,N$. The final point estimator is $\sqrt{2\bar{Z}}$, where $\bar{Z}=(1/N)\sum_{l=1}^NZ_l$, and the $1-\alpha$ confidence interval is $\sqrt{2}\times(\sqrt{\bar{Z}}\pm(\sigma/(2\sqrt{\bar{Z}}))t_{1-\alpha/2}/\sqrt{N})$ where $\sigma^2=1/(N-1)\sum_{n=1}^N(Z_n-\bar{Z})^2$ and $t_{1-\alpha/2}$ is the $1-\alpha/2$ percentile of the $t$-distribution with $N-1$ degree of freedom.
\end{enumerate}

This gives a consistent point estimate for $\sqrt{2Var_0(g(X))}$ and an asymptotically valid confidence interval. In our implementation we choose $K=100$, $n=50$ and $N=20$.

\section{Proofs} \label{sec:proofs}
\begin{proofof}{Proposition \ref{optimal L}}
We guess the solution \eqref{basic solution} by applying Euler-Lagrange equation and informally differentiate the integrand with respect to $L$. We will then verify rigorously that this candidate solution is indeed optimal.

Relaxing the constraint $E_0[L]=1$ in \eqref{max dual basic inner}, the objective becomes
$$E_0[h(X)L-\alpha L\log L+\lambda L-\lambda]$$
where $\lambda\in\mathbb{R}$ is the Lagrange multiplier. Treating $E_0[\cdot]$ as an integral, Euler-Lagrange equation implies that the derivative with respect to $L$ is
$$h(X)-\alpha\log L-\alpha+\lambda=0$$
which gives
$$\log L=\frac{h(X)}{\alpha}+\frac{\lambda-\alpha}{\alpha}$$
or that $L=\lambda'e^{h(X)/\alpha}$ for some $\lambda'>0$.
With the constraint that $E_0[L]=1$, a candidate solution is
\begin{equation}
L^*=\frac{e^{h(X)/\alpha}}{E_0[e^{h(X)/\alpha}]}. \label{L}
\end{equation}

To verify \eqref{L} formally, the following convexity argument will suffice. First, note that the objective value of \eqref{max dual basic inner} evaluated at $L^*$ given by \eqref{L} is
\begin{align}
E_0[h(X)L^*-\alpha L^*\log L^*]&=E_0\left[h(X)L^*-\alpha L^*\left(\frac{h(X)}{\alpha}-\log E_0[e^{h(X)/\alpha}]\right)\right] \notag\\
&=\alpha\log E_0[e^{h(X)/\alpha}]. \label{interim caculation}
\end{align}
Our goal is to show that
$$\alpha\log E_0[e^{h(X)/\alpha}]\geq E_0[h(X)L-\alpha L\log L]$$
for all $L\in\mathcal{L}$. Rearranging terms, this means we need
\begin{equation}
E_0[e^{h(X)/\alpha}]\geq e^{E_0[h(X)L-\alpha L\log L]/\alpha}. \label{target}
\end{equation}
To prove \eqref{target}, observe that, for any likelihood ratio $L$,
$$E_0[e^{h(X)/\alpha}]=E_0[LL^{-1}e^{h(X)/\alpha}]=E_0[Le^{h(X)/\alpha-\log L}]\geq e^{E_0[h(X)L/\alpha-L\log L]}$$
by using the convexity of the function $e^\cdot$ and Jensen's inequality over the expectation $E_0[L\ \cdot]$ in the last inequality. Note that equality holds if and only if $h(X)/\alpha-\log L$ is degenerate, i.e.~ $h(X)/\alpha-\log L=constant$, which reduces to $L^*$. Hence $L^*$ is the unique optimal solution for \eqref{max dual basic inner}.

In conclusion, when $1/\alpha\in\mathcal{D}^+:=\{\theta\in\mathbb{R}^+\setminus\{0\}:\psi(\theta)<\infty\}$ where $\psi(\theta)=\log E_0[e^{\theta h(X)}]$, the optimal solution of \eqref{max dual basic inner} is given by \eqref{basic solution}, with the optimal value $\alpha\log E_0[e^{h(X)/\alpha}]$.
\end{proofof}

\begin{proofof}{Lemma \ref{modification}}
For any $\eta\leq N$, we want to show that $E_0[L\log L]\leq\eta$ and $L\in\mathcal{L}$ together imply $L\in\mathcal{L}(M)$ for some large $M>0$. Note that $\Lambda(X)$ has exponential moment, since Holder's inequality implies
\begin{equation}
E_0[e^{\theta \Lambda(X)}]=E_0[e^{\theta\sum_{t=1}^T\Lambda_t(X)}]\leq\prod_{t=1}^T(E_0[e^{T\theta \Lambda_t(X)}])^{1/T}<\infty \label{q moment}
\end{equation}
when $\theta$ is small enough. Hence, for any $L\in\mathcal{L}$ that satisfies $E_0[L\log L]\leq\eta$, we have
$$E_0[e^{\theta \Lambda(X)}]=E_0[LL^{-1}e^{\theta \Lambda(X)}]=E_0[Le^{\theta \Lambda(X)-\log L}]<\infty$$
for small enough $\theta$, by \eqref{q moment}. Jensen's inequality implies that
$$e^{\theta E_0[\Lambda(X)L]-E_0[L\log L]}\leq E_0[Le^{\theta \Lambda(X)-\log L}]=E_0[e^{\theta \Lambda(X)}]<\infty.$$
Since $E_0[L\log L]\leq\eta\leq N$, we have $E_0[\Lambda(X)L]\leq M$ for some constant $M>0$. So $L\in\mathcal{L}(M)$. This concludes the lemma.
\end{proofof}

\subsection{Proofs in Section \ref{T general}}\label{sec:proof fixed point}
\begin{proofof}{Lemma \ref{lemma:multivariate fixed point}}
We prove the statement point-by-point regarding the operator $\mathcal{K}$. For convenience, denote $S_h(X,\mathbf{X}_{T-1})=S_h(X,X_1,X_2,\ldots,X_{T-1})$, where $S_h$ is defined in \eqref{bar S}, and $\underline{L}_{T-1}=\prod_{t=1}^{T-1}L_t(X_t)$ and $\underline{L}_{T-1}'=\prod_{t=1}^{T-1}L_t'(X_t)$. $X,X_1,\ldots,X_{T-1}$ are i.i.d.~random variables with distribution $P_0$. Also, denote $\beta=1/\alpha>0$, so $\beta\to0$ is equivalent to $\alpha\to\infty$. In this proof we let $C>0$ be a constant that can be different every time it shows up.
\\

\noindent\underline{Well-definedness and closedness}: Recall the definition of $K$ in \eqref{component K}, which can be written as
$$K(\mathbf{L})(x)=\frac{e^{\beta E_0[S_h(X,\mathbf{X}_{T-1})\underline{L}_{T-1}|X=x]/(T-1)!}}{E_0[e^{\beta E_0[S_h(X,\mathbf{X}_{T-1})\underline{L}_{T-1}|X]/(T-1)!}]}$$
for any $\mathbf{L}=(L_1,L_2,\ldots,L_{T-1})\in\mathcal{L}(M)^{T-1}$. We shall show that, for any $\mathbf{L}\in\mathcal{L}(M)^{T-1}$, we have $0<E_0[e^{\beta E_0[S_h(X,\mathbf{X}_{T-1})\underline{L}_{T-1}|X]/(T-1)!}]<\infty$ and that $K(\mathbf{L})\in\mathcal{L}(M)$. This will imply that, starting from any $L_1,L_2,\ldots,L_{T-1}\in\mathcal{L}(M)$, we get a well-defined operator $K$ and that  $\tilde{L}_1,\tilde{L}_2,\ldots,\tilde{L}_{T-1}$ defined in \eqref{relation K} all remain in $\mathcal{L}(M)$. We then conclude that $\mathcal K$ is both well-defined and closed in $\mathcal{L}(M)^{T-1}$ by the definition in \eqref{multivariate K}.

Now suppose $L_1,L_2,\ldots,L_{T-1}\in\mathcal{L}(M)$. Since $S_h(X,\mathbf X_{T-1})\leq(T-1)!\left(\Lambda(X)+\sum_{t=1}^{T-1}\Lambda(X_t)\right)$ by definition, we have
\begin{eqnarray}
E_0[e^{\beta E_0[S_h(X,\mathbf{X}_{T-1})\underline{L}_{T-1}|X]/(T-1)!}]&\leq&E_0[e^{\beta(\Lambda(X)+\sum_{t=1}^{T-1}E_0[\Lambda(X_t)L_t(X_t)])}] \notag\\
&=&E_0[e^{\beta \Lambda(X)}]e^{\beta\sum_{t=1}^{T-1}E_0[\Lambda(X)L_t(X)]} \notag\\
&\leq&E_0[e^{\beta \Lambda(X)}]e^{\beta(T-1)M} \notag\\
&<&\infty. \label{finite exponential moment gL}
\end{eqnarray}
This also implies that $e^{\beta E_0[S_h(X,\mathbf{X}_{T-1})\underline{L}_{T-1}|X]/(T-1)!}<\infty$ a.s.. Similarly,
\begin{equation}
E_0[e^{\beta E_0[S_h(X,\mathbf{X}_{T-1})\underline{L}_{T-1}|X]/(T-1)!}]\geq E_0[e^{-\beta \Lambda(X)}]e^{-\beta(T-1)M}>0. \label{lower bound q multivariate}
\end{equation}
Hence $K$ is well-defined. To show closedness, consider
\begin{align}
E_0[\Lambda(X)K(\mathbf{L}_{T-1})(X)]&=E_0\left[\Lambda(X)\frac{e^{\beta E_0[S_h(X,\mathbf{X}_{T-1})\underline{L}_{T-1}|X]/(T-1)!}}{E_0[e^{\beta E_0[S_h(X,\mathbf{X}_{T-1})\underline{L}_{T-1}|X]/(T-1)!}]}\right] \notag\\
&\leq\frac{E_0[\Lambda(X)e^{\beta \Lambda(X)}]e^{2\beta(T-1)M}}{E_0[e^{-\beta \Lambda(X)}]}. \label{closedness interim}
\end{align}
Since $E_0[\Lambda(X)e^{\beta \Lambda(X)}]\to E_0[\Lambda(X)]$ and $E_0[e^{-\beta \Lambda(X)}]\to1$ as $\beta\to0$, \eqref{closedness interim} is bounded by $M$ for small enough $\beta$, if we choose $M>E_0[\Lambda(X)]$. Hence $K$ is closed in $\mathcal{L}(M)$.

By recursing using \eqref{relation K}, we get that $\mathcal{K}$ is well-defined, and that for any $\mathbf{L}=(L_1,\ldots,L_{T-1})\in\mathcal{L}(M)^{T-1}$, we have $\max_{t=1,\ldots,T-1}E_0[\Lambda(X)\tilde{L}_t(X)]\leq M$, and so $\mathcal{K}$ is closed in $\mathcal{L}(M)^{T-1}$.
\\

\noindent\underline{Contraction}:
Consider, for any $\mathbf{L}=(L_1,\ldots,L_{T-1}),\mathbf{L}'=(L_1',\ldots,L_{T-1}')\in\mathcal{L}(M)^{T-1}$,
\begin{eqnarray}
&&E_0[(1+\Lambda(X))|K(\mathbf{L})(X)-K(\mathbf{L}')(X)|] \notag\\
&=&E_0\left[(1+\Lambda(X))\left|\frac{e^{\beta E_0[S_h(X,\mathbf{X}_{T-1})\underline{L}_{T-1}|X]/(T-1)!}}{E_0[e^{\beta E_0[S_h(X,\mathbf{X}_{T-1})\underline{L}_{T-1}|X]/(T-1)!}]}-\frac{e^{\beta E_0[S_h(X,\mathbf{X}_{T-1})\underline{L}_{T-1}'|X]/(T-1)!}}{E_0[e^{\beta E_0[S_h(X,\mathbf{X}_{T-1})\underline{L}_{T-1}'|X]/(T-1)!}]}\right|\right] \notag\\
&=&E_0\Bigg[(1+\Lambda(X))\Bigg|\frac{1}{\xi_2}(e^{\beta E_0[S_h(X,\mathbf{X}_{T-1})\underline{L}_{T-1}|X]/(T-1)!}-e^{\beta E_0[S_h(X,\mathbf{X}_{T-1})\underline{L}_{T-1}'|X]/(T-1)!}){} \notag\\
&&{}-\frac{\xi_1}{\xi_2^2}(E_0[e^{\beta E_0[S_h(X,\mathbf{X}_{T-1})\underline{L}_{T-1}|X]/(T-1)!}]-E_0[e^{\beta E_0[S_h(X,\mathbf{X}_{T-1})\underline{L}_{T-1}'|X]/(T-1)!}])\Bigg|\Bigg] \label{interim2}
\end{eqnarray}
by using mean value theorem, where $(\xi_1,\xi_2)$ lies in the line segment between $(e^{\beta E_0[S_h(X,\mathbf{X}_{T-1})\underline{L}_{T-1}|X]/(T-1)!},$ $E_0[e^{\beta E_0[S_h(X,\mathbf{X}_{T-1})\underline{L}_{T-1}|X]/(T-1)!}])$ and $(e^{\beta E_0[S_h(X,\mathbf{X}_{T-1})\underline{L}_{T-1}'|X]/(T-1)!},E_0[e^{\beta E_0[S_h(X,\mathbf{X}_{T-1})\underline{L}_{T-1}'|X]/(T-1)!}])$. By \eqref{lower bound q multivariate}, we have $\xi_2>1-\epsilon$ for some small $\epsilon>0$, when $\beta$ is small enough. Moreover, $\xi_1\leq e^{\beta(\Lambda(X)+(T-1)M)}$. Hence, \eqref{interim2} is less than or equal to
\begin{eqnarray}
&&E_0\Bigg[(1+\Lambda(X))\Bigg(\sup\left|\frac{1}{\xi_2}\right|\left|e^{\beta E_0[S_h(X,\mathbf{X}_{T-1})\underline{L}_{T-1}|X]/(T-1)!}-e^{\beta E_0[S_h(X,\mathbf{X}_{T-1})\underline{L}_{T-1}'|X]/(T-1)!}\right|{} \notag\\
&&{}+\sup\left|\frac{\xi_1}{\xi_2^2}\right|\left|E_0[e^{\beta E_0[S_h(X,\mathbf{X}_{T-1})\underline{L}_{T-1}|X]/(T-1)!}]-E_0[e^{\beta E_0[S_h(X,\mathbf{X}_{T-1})\underline{L}_{T-1}'|X]/(T-1)!}]\right|\Bigg)\Bigg] \notag\\
&\leq&\frac{1}{1-\epsilon}E_0\left[(1+\Lambda(X))\left|e^{\beta E_0[S_h(X,\mathbf{X}_{T-1})\underline{L}_{T-1}|X]/(T-1)!}-e^{\beta E_0[S_h(X,\mathbf{X}_{T-1})\underline{L}_{T-1}'|X]/(T-1)!}\right|\right]{} \notag\\
&&{}+\frac{e^{\beta(T-1)M}}{(1-\epsilon)^2}E_0\Big[(1+\Lambda(X))e^{\beta \Lambda(X)}\Big|E_0[e^{\beta E_0[S_h(X,\mathbf{X}_{T-1})\underline{L}_{T-1}|X]/(T-1)!}]{} \notag\\
&&{}-E_0[e^{\beta E_0[S_h(X,\mathbf{X}_{T-1})\underline{L}_{T-1}'|X]/(T-1)!}]\Big|\Big] \notag\\
&\leq&CE_0\left[(1+\Lambda(X))(e^{\beta \Lambda(X)}+1)\left|e^{\beta E_0[S_h(X,\mathbf{X}_{T-1})\underline{L}_{T-1}|X]/(T-1)!}-e^{\beta E_0[S_h(X,\mathbf{X}_{T-1})\underline{L}_{T-1}'|X]/(T-1)!}\right|\right] \notag\\
&\leq&C\beta E_0\left[(1+\Lambda(X))\frac{e^{2\beta \Lambda(X)}}{(T-1)!}|E_0[S_h(X,\mathbf{X}_{T-1})\underline{L}_{T-1}|X]-E_0[S_h(X,\mathbf{X}_{T-1})\underline{L}_{T-1}'|X]|\right] \notag\\
&&\text{\ \ by mean value theorem again} \notag\\
&\leq&C\beta E_0\left[(1+\Lambda(X))\frac{e^{2\beta \Lambda(X)}}{(T-1)!}|S_h(X,\mathbf{X}_{T-1})||\underline{L}_{T-1}-\underline{L}_{T-1}'|\right] \notag\\
&\leq&C\beta E_0\left[(1+\Lambda(X))e^{2\beta \Lambda(X)}\left(\Lambda(X)+\sum_{t=1}^{T-1}\Lambda(X_t)\right)|\underline{L}_{T-1}-\underline{L}_{T-1}'|\right] \notag\\
&=&C\beta\Bigg(E_0[(1+\Lambda(X))^2e^{2\beta \Lambda(X)}]E_0|\underline{L}_{T-1}-\underline{L}_{T-1}'|+E_0[(1+\Lambda(X))e^{2\beta \Lambda(X)}]{} \notag\\
&&{}E_0\left[\sum_{t=1}^{T-1}\Lambda(X_t)|\underline{L}_{T-1}-\underline{L}_{T-1}'|\right]\Bigg) \label{contraction interim}
\end{eqnarray}
when $\beta$ is small enough (or $\alpha$ is large enough). Now note that
$$|\underline{L}_{T-1}-\underline{L}_{T-1}'|\leq\sum_{s=1}^{T-1}\dot{\underline{L}}_{T-1}^s|L_s(X_s)-L_s'(X_s)|$$
where each $\dot{\underline{L}}_{T-1}^s$ is a product of either one of $L(X_r)$ or $L'(X_r)$ for $r=1,\ldots,T-1,r\neq s$. Hence \eqref{contraction interim} is less than or equal to
\begin{eqnarray}
&&C\beta\Bigg(E_0[(1+\Lambda(X))^2e^{2\beta \Lambda(X)}]\sum_{s=1}^{T-1}E_0|L_s(X)-L_s'(X)|+E_0[(1+\Lambda(X))e^{2\beta \Lambda(X)}]{} \notag\\
&&{}\sum_{t=1}^{T-1}\left(E_0[\Lambda(X)|L_t(X)-L_t'(X)|]+E_0[\Lambda(X)L(X)]\sum_{\substack{s=1,\ldots,T-1\\s\neq t}}E_0|L_s(X)-L_s'(X)|\right)\Bigg). \label{interim L1}
\end{eqnarray}
Now for convenience denote $y_t=E_0[(1+\Lambda(X))|L_t(X)-L_t'(X)|]$ and $\tilde{y}_t=E_0[(1+\Lambda(X))|\tilde{L}_t(X)-\tilde{L}_t'(X)|]$, where $\tilde{L}_t$ is defined in \eqref{relation K}. Also denote $\mathbf{y}=(y_1,\ldots,y_{T-1})'$ and $\tilde{\mathbf{y}}=(\tilde{y}_1,\ldots,\tilde{y}_{T-1})'$, where $'$ denotes transpose. Then \eqref{interim L1} gives $\tilde{y}_1\leq a\mathbf{1}'\mathbf{y}$ for some $a:=a(\beta)=O(\beta)$ as $\beta\to0$, and $\mathbf{1}$ denotes the $(T-1)$-dimensional vector of constant 1. Then, by iterating the same argument as above, \eqref{relation K} implies that
\begin{align*}
\tilde{y}_2&\leq a\mathbf{1}'(\tilde{y}_1,y_2,\ldots,y_{T-1})'\\
\tilde{y}_3&\leq a\mathbf{1}'(\tilde{y}_1,\tilde{y}_2,y_3,\ldots,y_{T-1})'\\
&\vdots\\
\tilde{y}_{T-1}&\leq a\mathbf{1}'(\tilde{y}_1,\ldots,\tilde{y}_{T-2},y_{T-1})'.
\end{align*}
Hence
$$\max_{t=1,\ldots,T-1}\tilde{y}_t\leq\max_{t=1,\ldots,T-1}\sum_{s=1}^TA_{ts}y_s$$
where $A_{ts}:=A_{ts}(\beta)$ are constants that go to 0 as $\beta\to0$. Therefore, when $\beta$ is small enough, $d(\mathcal K(\mathbf{L}_{T-1}),\mathcal K(\mathbf{L}_{T-1}'))\leq wd(\mathbf{L}_{T-1},\mathbf{L}_{T-1}')$ for some $0<w<1$, and $\mathcal K$ is a contraction. By Banach fixed point theorem, there exists a unique fixed point $\mathbf L^*$. Moreover, as a consequence, starting from any initial value $\mathbf{L}^{(1)}=\mathbf{L}\in\mathcal{L}(M)^{T-1}$, the recursion $\mathbf{L}^{(k+1)}=\mathcal K(\mathbf{L}^{(k)})$ satisfies $\mathbf{L}^{(k)}\stackrel{d}{\to}\mathbf{L}^*$ where $\mathbf{L}^*$ is the fixed point of $\mathcal K$.
\\

\noindent\underline{Identical Components}:
It remains to show that all components of $\mathbf{L}^*$ are the same. Denote $\mathbf{L}^*=(L_1^*,\ldots,L_{T-1}^*)$. By definition $\mathcal K(\mathbf{L}^*)=\mathbf{L}^*$. So, using \eqref{relation K}, we have
\begin{align*}
\tilde{L}_1&=K(L_1^*,\ldots,L_{T-1}^*)=L_1^*\\
\tilde{L}_2&=K(\tilde{L}_1,L_2^*,\ldots,L_{T-1}^*)=K(L_1^*,L_2^*,\ldots,L_{T-1}^*)=L_2^*\\
\tilde{L}_3&=K(\tilde{L}_1,\tilde{L}_2,L_3^*,\ldots,L_{T-1}^*)=K(L_1^*,L_2^*,L_3^*,\ldots,L_{T-1}^*)=L_3^*\\
&\vdots\\
\tilde{L}_{T-1}&=K(\tilde{L}_1,\ldots,\tilde{L}_{T-2},L_{T-1}^*)=K(L_1^*,\ldots,L_T^*,L_{T-1}^*)=L_{T-1}^*.
\end{align*}
Hence $L_1^*=L_2^*=\cdots=L_{T-1}^*=\tilde{L}_1=\cdots=\tilde{L}_{T-1}=K(L_1^*,\ldots,L_{T-1}^*)$. This concludes the lemma.

\end{proofof}

\begin{proofof}{Corollary \ref{corollary:component}}
By Lemma \ref{lemma:multivariate fixed point}, $\mathcal{K}$ has a fixed point in $\mathcal{L}(M)^{T-1}$ that has all equal components. Since convergence in $\mathcal{L}(M)^{T-1}$ implies convergence in each component in $\mathcal{L}(M)$ (in the $\|\cdot\|_\Lambda$-norm), and that by construction $K(L_1,L_2,\ldots,L_{T-1})=K(L_{T-1},L_1,L_2,\ldots,L_{T-2})$ for any $L_1,L_2,\ldots,L_{T-1}\in\mathcal L(M)$, the result follows.
\end{proofof}

\begin{proofof}{Lemma \ref{lemma:multivariate monotonicity}}
Consider
\begin{eqnarray*}
&&E_0\left[S_h(\mathbf{X}_T)\prod_{t=1}^TL^{(k+t-1)}(X_t)\right]-\alpha(T-1)!\sum_{t=1}^TE_0[L^{(k+t-1)}\log L^{(k+t-1)}]\\
&=&E_0\left[E_0\left[S_h(\mathbf{X}_T)\prod_{t=2}^TL^{(k+t-1)}(X_t)\Bigg|X_1\right]L^{(k)}(X_1)\right]-\alpha (T-1)!E_0[L^{(k)}\log L^{(k)}]{}\\
&&{}-\alpha(T-1)!\sum_{t=2}^TE_0[L^{(k+t-1)}\log L^{(k+t-1)}]\\
&\leq&E_0\left[E_0\left[S_h(\mathbf{X}_T)\prod_{t=2}^TL^{(k+t-1)}(X_t)\Bigg|X_1\right]L^{(k+T)}(X_1)\right]-\alpha(T-1)! E_0[L^{(k+T)}\log L^{(k+T)}]{}\\
&&{}-\alpha(T-1)!\sum_{t=2}^TE_0[L^{(k+t-1)}\log L^{(k+t-1)}]\\
&=&E_0\left[S_h(\mathbf{X}_T)\prod_{t=1}^TL^{(k+t)}(X_t)\right]-\alpha(T-1)!\sum_{t=1}^TE_0[L^{(k+t)}\log L^{(k+t)}].
\end{eqnarray*}
The inequality holds because $L^{(k+T)}=K(L^{(k+1)},\ldots,L^{(k+T-1)})$, which, by Proposition \ref{optimal L} and the definition of $K$, maximizes the objective $E_0\left[E_0\left[S_h(\mathbf{X}_T)\prod_{t=2}^TL^{(k+t-1)}(X_t)\Bigg|X_1\right]L(X_1)\right]-\alpha(T-1)!E_0[L\log L]$ over $L$. The last equality can be seen by the invariance of $S_h$ over permutations of its arguments, and relabeling $X_2$ by $X_1$, $X_3$ by $X_2$, up to $X_T$ by $X_{T-1}$ and $X_1$ by $X_T$.
\end{proofof}

\begin{proofof}{Lemma \ref{lemma:multivariate convergence}}
We consider convergence of the first and the second terms of \eqref{convergence multivariate} separately. For the first term, consider
\begin{eqnarray}
&&\left|E_0\left[S_h(\mathbf{X}_T)\prod_{t=1}^TL^{(k+t-1)}(X_t)\right]-E_0\left[S_h(\mathbf{X}_T)\prod_{t=1}^TL^*(X_t)\right]\right| \notag\\
&\leq&E_0\left[|S_h(\mathbf{X}_T)|\left|\prod_{t=1}^TL^{(k+t-1)}(X_t)-\prod_{t=1}^TL^*(X_t)\right|\right] \notag\\
&\leq&(T-1)!E_0\left[\sum_{t=1}^T\Lambda(X_t)\sum_{s=1}^T\dot{\underline{L}}_T^s|L^{(k+s-1)}(X_s)-L^*(X_s)|\right] \notag\\
&&\text{\ \ where each $\dot{\underline{L}}_T^s$ is product of either one of $L^{(k+r-1)}(X_r)$ or $L^*(X_r)$ for $r=1,\ldots,T,\ r\neq s$} \notag\\
&\leq&C\sum_{s=1}^TE_0[(1+\Lambda(X))|L^{(k+s-1)}(X)-L^*(X)|]\text{\ \ for some constant $C>0$} \notag\\
&\to&0 \label{multivariate convergence1}
\end{eqnarray}
as $k\to\infty$, since $L^{(k)}\to L^*$ in $\|\cdot\|_\Lambda$-norm by Corollary \ref{corollary:component}.

We now consider the second term in \eqref{convergence multivariate}. By the recursion of $K$, we have, for $k\geq 1$,
\begin{eqnarray}
&&|E_0[L^{(k+T-1)}\log L^{(k+T-1)}]-E_0[L^*\log L^*]| \notag\\
&=&\Bigg|\left(\frac{\beta}{(T-1)!}E_0\left[S_h(\mathbf{X}_T)\prod_{t=1}^TL^{(k+t-1)}(X_t)\right]-\log E_0[e^{\beta E_0[S_h(X,\mathbf{X}_{T-1})\prod_{t=1}^{T-1}L^{(k+t-1)}(X_t)|X]/(T-1)!}]\right){} \notag\\
&&{}-\left(\frac{\beta}{(T-1)!}E_0\left[S_h(\mathbf{X}_T)\prod_{t=1}^TL^*(X_t)\right]-\log E_0[e^{\beta E_0[S_h(X,\mathbf{X}_{T-1})\prod_{t=1}^{T-1}L^*(X_t)|X]/(T-1)!}]\right)\Bigg| \notag\\
&=&\frac{\beta}{(T-1)!}\left|E_0\left[S_h(\mathbf{X}_T)\prod_{t=1}^TL^{(k+t-1)}(X_t)\right]-E_0\left[S_h(\mathbf{X}_T)\prod_{t=1}^TL^*(X_t)\right]\right|{} \notag\\
&&{}+|\log E_0[e^{\beta E_0[S_h(X,\mathbf{X}_{T-1})\prod_{t=1}^{T-1}L^{(k+t-1)}(X_t)|X]/(T-1)!}]-\log E_0[e^{\beta E_0[S_h(X,\mathbf{X}_{T-1})\prod_{t=1}^{T-1}L^*(X_t)|X]/(T-1)!}]|. \label{interim KL multivariate}
\end{eqnarray}
The first term in \eqref{interim KL multivariate} converges to 0 by the same argument as in \eqref{multivariate convergence1}. For the second term, we can write, by mean value theorem, that
\begin{eqnarray}
&&|\log E_0[e^{\beta E_0[S_h(X,\mathbf{X}_{T-1})\prod_{t=1}^{T-1}L^{(k+t-1)}(X_t)|X]/(T-1)!}]-\log E_0[e^{\beta E_0[S_h(X,\mathbf{X}_{T-1})\prod_{t=1}^{T-1}L^*(X_t)|X]/(T-1)!}]| \notag\\
&=&\frac{1}{\xi_1}|E_0[e^{\beta E_0[S_h(X,\mathbf{X}_{T-1})\prod_{t=1}^{T-1}L^{(k+t-1)}(X_t)|X]/(T-1)!}]-E_0[e^{\beta E_0[S_h(X,\mathbf{X}_{T-1})\prod_{t=1}^{T-1}L^*(X_t)|X]/(T-1)!}]| \label{interim2 KL multivariate}
\end{eqnarray}
where $\xi_1$ lies between $E_0[e^{\beta E_0[S_h(X,\mathbf{X}_{T-1})\prod_{t=1}^{T-1}L^{(k+t-1)}(X_t)|X]/(T-1)!}]$ and $E_0[e^{\beta E_0[S_h(X,\mathbf{X}_{T-1})\prod_{t=1}^{T-1}L^*(X_t)|X]/(T-1)!}]$, and hence $\xi_1\geq E_0[e^{-\beta \Lambda(X)}]e^{-\beta E_0[\Lambda(X)L(X)]}\geq1-\epsilon$ for some small $\epsilon>0$, when $\beta$ is small enough, by a similar argument as in the proof of Lemma \ref{lemma:multivariate fixed point}. Moreover,
\begin{eqnarray*}
&&|E_0[e^{\beta E_0[S_h(X,\mathbf{X}_{T-1})\prod_{t=1}^{T-1}L^{(k+t-1)}(X_t)|X]/(T-1)!}]-E_0[e^{\beta E_0[S_h(X,\mathbf{X}_{T-1})\prod_{t=1}^{T-1}L^*(X_t)|X]/(T-1)!}]|{}\\
&\leq&\beta E_0\left[e^{\beta \xi_2}|S_h(X,\mathbf{X}_{T-1})|\left|\prod_{t=1}^{T-1}L^{(k+t-1)}(X_t)-\prod_{t=1}^{T-1}L^*(X_t)\right|\right]
\end{eqnarray*}
for some $\xi_2$ lying between $E_0[S_h(X,\mathbf{X}_{T-1})\prod_{t=1}^{T-1}L^{(k+t-1)}(X_t)|X]/(T-1)!$ and $E_0[S_h(X,\mathbf{X}_{T-1})\prod_{t=1}^{T-1}L^*(X_t)|X]/(T-1)!$. Hence, much like the argument in proving the contraction property in Lemma \ref{lemma:multivariate fixed point}, we have $\xi_2\leq \Lambda(X)+(T-1)M$ and \eqref{interim2 KL multivariate} is less than or equal to
$$C\beta\max_{t=1,\ldots,T-1}E_0[(1+\Lambda(X))|L^{(k+t-1)}(X)-L^*(X)|]\to0$$
as $k\to\infty$ for some $C>0$. This concludes the lemma.
\end{proofof}

\subsection{Proof of Theorem \ref{thm:fixed time}}\label{sec:proof asymptotics}
For convenience, denote $\beta=1/\alpha^*>0$, so $\beta$ is small when $\alpha^*$ is large. Also let $X$ be a generic random variable with distribution $P_0$. Then from \eqref{optimal L fixed time} we have
\begin{equation}
L^*(x)=\frac{e^{\beta g^{L^*}(x)}}{E_0[e^{\beta g^{L^*}(X)}]} \label{L recap}
\end{equation}
where $g^{L^*}(x)=\sum_{t=1}^Tg_t^{L^*}(x)=\sum_{t=1}^TE_0[h(\mathbf{X}_T)\prod_{\substack{1\leq r\leq T\\r\neq t}}L^*(X_r)|X_t=x]$. Also recall that
$$g(x)=\sum_{t=1}^Tg_t(x)=\sum_{t=1}^TE_0[h(\mathbf{X}_T)|X_t=x]$$
as defined in \eqref{gt}, so that $E_0[g(X)]=TE_0[h(\mathbf{X}_T)]$.
Furthermore, let us denote, for any $p\geq1$, $\bar{O}(\beta^p):=\bar{O}(\beta^p;x)$ as a deterministic function in $x$ such that $E_0[h(\mathbf{X}_T)^q\bar{O}(\beta^p;X_t)]=O(\beta^p)$ for any $q\geq1$ and $t=1,\ldots,T$, when $\beta\to0$. Finally, we also let $\psi_L(\beta):=\log E_0[e^{\beta g^L(X)}]$ for convenience.

We first give a quadratic approximation of $L^*$ as $\beta \to0$ (equivalently $\alpha^*\to\infty$). Then we find the relation between $\beta $ and $\eta$, which verifies the optimality condition given in Theorem \ref{nonconvex}. After that we expand the objective value in terms of $\beta $, and hence $\eta$, to conclude Theorem \ref{thm:fixed time}.
\\

\noindent\textbf{Asymptotic expansion of $L^*$:} We shall obtain a quadratic approximation of $L^*$ by first getting a first order approximation of $L^*$ and then iterating via the quantity $g^{L^*}$ to get to the second order. Note that as the logarithmic moment generating function of $g^{L^*}(X)$,
\begin{align}
\psi_{L^*}(\beta )&=\log E_0[e^{\beta g^{L^*}(X)}] \notag\\
&=\beta E_0[g^{L^*}(X)]+\frac{\beta^ 2}{2}\kappa_2(g^{L^*}(X))+\frac{\beta^ 3}{3!}\kappa_3(g^{L^*}(X))+O(\beta^ 4) \label{psi}
\end{align}
where $\kappa_2(g^{L^*}(X)):=E_0[(g^{L^*}(X)-E_0[g^{L^*}(X)])^2]$ and $\kappa_3(g^{L^*}(X)):=[(g^{L^*}(X)-E_0[g^{L^*}(X)])^3]$.
Using \eqref{L recap} and \eqref{psi}, and the finiteness of the exponential moment of $g^{L^*}(X)$ guaranteed by a calculation similar to \eqref{finite exponential moment gL}, we have
\begin{align}
L^*(x)&=\frac{e^{\beta g^{L^*}(x)}}{E_0[e^{\beta g^{L^*}(X)}]}=e^{\beta g^{L^*}(x)-\psi_{L^*}(\beta )} \notag\\
&=1+\beta (g^{L^*}(x)-E_0[g^{L^*}(X)])+\bar{O}(\beta^ 2). \label{L1}
\end{align}
But notice that
\begin{align*}
g^{L^*}(x)&=\sum_{t=1}^TE_0\left[h(\mathbf{X}_T)\prod_{\substack{1\leq r\leq T\\r\neq t}}L^*(X_r)\Bigg|X_t=x\right]\\
&=\sum_{t=1}^TE_0\left[h(\mathbf{X}_T)\prod_{\substack{1\leq r\leq T\\r\neq t}}(1+\bar{O}(\beta;X_r ))\Bigg|X_t=x\right]\\
&=\sum_{t=1}^TE_0[h(\mathbf{X}_T)|X_t=x]+\bar{O}(\beta )\\
&=g(x)+\bar{O}(\beta )
\end{align*}
and hence $E_0[g^{L^*}(X)]=E_0[g(X)]+O(\beta )$. Consequently, from \eqref{L1} we have
\begin{equation}
L^*(x)=1+\beta (g(x)-E_0[g(X)])+\bar{O}(\beta^ 2). \label{L first order}
\end{equation}
This gives a first order approximation of $L^*$. Using \eqref{L first order}, we strengthen our approximation of $g^{L^*}$ to get
\begin{align}
g^{L^*}(x)&=\sum_{t=1}^TE_0\left[h(\mathbf{X}_T)\prod_{\substack{1\leq r\leq T\\r\neq t}}(1+\beta (g(X_r)-E_0[g(X)])+\bar{O}(\beta^ 2))\Bigg|X_t=x\right] \notag\\
&=g(x)+\beta \sum_{t=1}^T\sum_{\substack{1\leq r\leq T\\r\neq t}}E_0[h(\mathbf{X}_T)(g(X_r)-E_0[g(X)])|X_t=x]+\bar O(\beta^ 2) \notag\\
&=g(x)+\beta W(x)+\bar O(\beta^ 2) \label{gL first order}
\end{align}
where we define $W(x):=\sum_{t=1}^T\sum_{\substack{1\leq r\leq T\\r\neq t}}E_0[h(\mathbf{X}_T)(g(X_r)-E_0[g(X)])|X_t=x]$. With \eqref{gL first order}, and using \eqref{psi} again, we then strengthen the approximation of $L^*$ to get
\begin{eqnarray}
L^*(x)&=&e^{\beta g^{L^*}(x)-\psi_{L^*}(\beta )}=e^{\beta (g^{L^*}(x)-E_0[g^{L^*}(X)])-\frac{\beta^ 2}{2}E_0[(g^{L^*}(X)-E_0[ g^{L^*}(X)])^2]+\bar{O}(\beta^ 3)} \notag\\
&=&1+\beta (g^{L^*}(x)-E_0[g^{L^*}(X)])+\frac{\beta^ 2}{2}[(g^{L^*}(x)-E_0[g^{L^*}(X)])^2-E_0[(g^{L^*}(X)-E_0[ g^{L^*}(X)])^2]]{} \notag\\
&&{}+\bar{O}(\beta^ 3) \notag\\
&=&1+\beta (g(x)-E_0[g(X)])+\beta^ 2\Big[W(x)-E_0[W(X)]+\frac{1}{2}((g(x)-E_0[g(X)])^2{} \notag\\
&&{}-E_0[(g(X)-E_0[g(X)])^2])\Big]+\bar{O}(\beta^ 3) \notag\\
&=&1+\beta (g(x)-E_0[g(X)])+\beta^ 2V(x)+\bar{O}(\beta^ 3) \label{L second order}
\end{eqnarray}
where we define $V(x):=W(x)-E_0[W(X)]+\frac{1}{2}((g(x)-E_0[g(X)])^2-E_0[(g(X)-E_0[g(X)])^2])$.
\\


\noindent\textbf{Relation between $\beta $ and $\eta$:} By substituting $L^*$ depicted in \eqref{L recap} into $\eta=E_0[L^*\log L^*]$, we have
\begin{equation}
\eta=E_0[L^*\log L^*]=\beta E_0[g^{L^*}(X)L^*(X)]-\log E_0[e^{\beta g^{L^*}(X)}]=\beta TE_0[h(\mathbf{X}_T)\underline{L}_T^*]-\psi_{L^*}(\beta ) \label{eta relation}
\end{equation}
Using \eqref{psi}, we can write \eqref{eta relation} as
\begin{eqnarray}
&&\beta TE_0[h(\mathbf{X}_T)\underline{L}_T^*]-\beta E_0[g^{L^*}(X)]-\frac{\beta^ 2}{2}\kappa_2(g^{L^*}(X))-\frac{\beta^ 3}{3!}\kappa_3(g^{L^*}(X))+O(\beta^ 4) \notag\\
&=&\beta \sum_{t=1}^TE_0\left[h(\mathbf{X}_T)\prod_{\substack{1\leq r\leq T\\r\neq t}}L^*(X_r)(L^*(X_t)-1)\right]-\frac{\beta^ 2}{2}\kappa_2(g^{L^*}(X))-\frac{\beta^ 3}{3!}\kappa_3(g^{L^*}(X)){} \notag\\
&&{}+O(\beta^ 4). \label{eta relation1}
\end{eqnarray}
We analyze \eqref{eta relation1} term by term. For the first term, using \eqref{L second order}, we have
\begin{eqnarray}
&&\sum_{t=1}^TE_0\left[h(\mathbf{X}_T)\prod_{\substack{1\leq r\leq T\\r\neq t}}L^*(X_r)(L^*(X_t)-1)\right] \notag\\
&=&\sum_{t=1}^TE_0\Bigg[h(\mathbf{X}_T)\prod_{\substack{1\leq r\leq T\\r\neq t}}\left(1+\beta (g(X_r)-E_0[g(X)])+\beta^ 2V(X_r)+\bar{O}(\beta^ 3)\right){} \notag\\
&&{}\cdot\left(\beta (g(X_t)-E_0[g(X)])+\beta^ 2V(X_t)+\bar{O}(\beta^ 3)\right)\Bigg] \notag\\
&=&\beta \sum_{t=1}^TE_0[h(\mathbf{X}_T)(g(X_t)-E_0[g(X)])]+\beta^ 2\Bigg[\sum_{t=1}^T\sum_{\substack{1\leq r\leq T\\r\neq t}}E_0[h(\mathbf{X}_T)(g(X_r)-E_0[g(X)])(g(X_t)-E_0[g(X)])]{} \notag\\
&&{}+\sum_{t=1}^TE_0[h(\mathbf{X}_T)V(X_t)]\Bigg]+O(\beta^ 3) \notag\\
&=&\beta Var_0(g(X))+\beta^ 2[\nu+E_0[g(X)V(X)]]+O(\beta^ 3) \label{first term}
\end{eqnarray}
where $\nu$ is defined in \eqref{nu}. The last equality follows since
\begin{eqnarray*}
&&\sum_{t=1}^TE_0[h(\mathbf{X}_T)(g(X_t)-E_0[g(X)])]=\sum_{t=1}^TE_0[E_0[h(\mathbf{X}_T)|X_t](g(X_t)-E_0[g(X)])]\\
&=&\sum_{t=1}^TE_0[g_t(X)(g(X)-E_0[g(X)])]=E_0[g(X)(g(X)-E_0[g(X)])]=Var_0(g(X)),
\end{eqnarray*}
\begin{eqnarray*}
&&\sum_{t=1}^T\sum_{\substack{1\leq r\leq T\\r\neq t}}E_0[h(\mathbf{X}_T)(g(X_r)-E_0[g(X)])(g(X_t)-E_0[g(X)])]\\
&=&\sum_{t=1}^T\sum_{\substack{1\leq r\leq T\\r\neq t}}E_0[E_0[h(\mathbf{X}_T)|X_r,X_t](g(X_r)-E_0[g(X)])(g(X_t)-E_0[g(X)])]\\
&=&E_0[G(X,Y)(g(X)-E_0[g(X)])(g(Y)-E_0[g(Y)])]=\nu
\end{eqnarray*}
where $G(X,Y)$ is defined in \eqref{G definition}, and
$$\sum_{t=1}^TE_0[h(\mathbf{X}_T)V(X_t)]=\sum_{t=1}^TE_0[E_0[h(\mathbf{X}_T)|X_t]V(X_t)]=E_0[g(X)V(X)].$$
For the second term in \eqref{eta relation1}, by using \eqref{gL first order}, we have
\begin{align}
\kappa_2(g^{L^*}(X))&=E_0[(g^{L^*}(X)-E_0[g^{L^*}(X)])^2] \notag\\
&=E_0[((g(X)-E_0[g(X)])+\beta (W(X)-E_0[W(X)])+\bar{O}(\beta^ 2;X))^2] \notag\\
&=Var_0(g(X))+2\beta E_0[(g(X)-E_0[g(X)])(W(X)-E_0[W(X)])]+O(\beta^ 2). \label{kappa 2 eta}
\end{align}
Now notice that $W(x)$ can be written as
\begin{align*}
W(x)&=\sum_{t=1}^T\sum_{\substack{1\leq r\leq T\\r\neq t}}E_0[h(\mathbf{X}_T)(g(X_r)-E_0[g(X)])|X_t=x]\\
&=\sum_{t=1}^T\sum_{\substack{1\leq r\leq T\\r\neq t}}E_0[E_0[h(\mathbf{X}_T)|X_r,X_t](g(X_r)-E_0[g(X)])|X_t=x]\\
&=E_0[G(X,Y)(g(Y)-E_0[g(Y)])|X=x]
\end{align*}
where $G(X,Y)$ is defined in \eqref{G definition}. Hence
\begin{eqnarray*}
&&E_0[(g(X)-E_0[g(X)])(W(X)-E_0[W(X)])]=E_0[(g(X)-E_0[g(X)])W(X)]\\
&=&E_0[(g(X)-E_0[g(X)])G(X,Y)(g(Y)-E_0[g(Y)])]=\nu.
\end{eqnarray*}
Consequently, \eqref{kappa 2 eta} becomes
\begin{equation}
Var_0(g(X))+2\beta \nu+O(\beta^ 2). \label{kappa 2 eta1}
\end{equation}
Finally, for the third term in \eqref{eta relation1}, we have
\begin{equation}
\kappa_3(g^{L^*}(X))=E_0[(g(X)-E_0[g(X)])^3]+O(\beta )=\kappa_3(g(X))+O(\beta ).
\label{kappa 3 eta}
\end{equation}
Combining \eqref{first term}, \eqref{kappa 2 eta1} and \eqref{kappa 3 eta}, we have
\begin{align}
\eta&=\beta^ 2Var_0(g(X))+\beta^ 3[\nu+E_0[g(X)V(X)]]-\frac{\beta^ 2}{2}Var_0(g(X))-\beta^ 3\nu-\frac{\beta^ 3}{6}\kappa_3(g(X))+O(\beta^ 4) \notag\\
&=\frac{\beta^ 2}{2}Var_0(g(X))+\beta^ 3\left[E_0[g(X)V(X)]-\frac{1}{6}\kappa_3(g(X))\right]+O(\beta^ 4). \label{eta relation2}
\end{align}
Under Assumption \ref{not sup}, and by routinely checking that the term $O(\beta^ 4)$ in \eqref{eta relation2} above is continuous in $\beta$, we can invert \eqref{eta relation2} to get
\begin{align}
\beta &=\sqrt{\frac{2\eta}{Var_0(g(X))}}\left(1+\frac{2\beta (E_0[g(X)V(X)]-(1/6)\kappa_3(g(X)))}{Var_0(g(X))}+O(\beta^ 2)\right)^{-1/2} \notag\\
&=\sqrt{\frac{2\eta}{Var_0(g(X))}}-\frac{1}{2}\sqrt{\frac{2\eta}{Var_0(g(X))}}\ \frac{2\beta (E_0[g(X)V(X)]-(1/6)\kappa_3(g(X)))}{Var_0(g(X))}+O(\eta^{1/2}\beta^ 2) \notag\\
&=\sqrt{\frac{2\eta}{Var_0(g(X))}}-\frac{2\eta(E_0[g(X)V(X)]-(1/6)\kappa_3(g(X)))}{(Var_0(g(X)))^2}+O(\eta^{3/2}). \label{eta relation3}
\end{align}
This in particular verifies the condition in Theorem \ref{nonconvex}, i.e.~for any small $\eta$, there exists a large enough $\alpha^*>0$ and a corresponding $L^*$ that satisfies \eqref{optimality}. This $L^*$ is an optimal solution to \eqref{max L}.
\\

\noindent\textbf{Relation between the objective value and $\beta $, and hence $\eta$:} Using \eqref{L second order} again, the optimal objective value in \eqref{max L} can be written as
\begin{eqnarray}
&&E_0[h(\mathbf{X}_T)\underline{L}_T^*] \notag\\
&=&E_0\left[h(\mathbf{X}_T)\prod_{t=1}^T(1+\beta (g(X_t)-E_0[g(X)])+\beta^ 2V(X_t)+\bar{O}(\beta^ 3;X_t))\right] \notag\\
&=&E_0[h(\mathbf{X}_T)]+\beta \sum_{t=1}^TE_0[h(\mathbf{X}_T)(g(X_t)-E_0[g(X)])]{} \notag\\
&&{}+\beta^ 2\left[\sum_{t=1}^T\sum_{\substack{1\leq r\leq T\\r<t}}E_0[h(\mathbf{X}_T)(g(X_r)-E_0[g(X)])(g(X_t)-E_0[g(X)])]+\sum_{t=1}^TE_0[h(\mathbf{X}_T)V(X_t)]\right]+O(\beta^ 3) \notag\\
&=&E_0[h(\mathbf{X}_T)]+\beta Var_0(g(X))+\beta^ 2\left[\frac{\nu}{2}+E_0[g(X)V(X)]\right]+O(\beta^ 3) \label{objective eta}
\end{eqnarray}
where the last equality follows from similar argument in \eqref{first term}. Finally, substituting \eqref{eta relation3} into \eqref{objective eta} gives
\begin{eqnarray*}
&&E_0[h(\mathbf{X}_T)]+\sqrt{2Var_0(g(X))\eta}+\frac{2\eta}{Var_0(g(X))}\left[-E_0[g(X)V(X)]+\frac{1}{6}\kappa_3(g(X))+\frac{\nu}{2}+E_0[g(X)V(X)]\right]{}\\
&&{}+O(\eta^{3/2})\\
&=&E_0[h(\mathbf{X}_T)]+\sqrt{2Var_0(g(X))\eta}+\frac{\eta}{Var_0(g(X))}\left[\frac{1}{3}\kappa_3(g(X))+\nu\right]+O(\eta^{3/2})
\end{eqnarray*}
which coincides with Theorem \ref{thm:fixed time}.

\subsection{Proof of Proposition \ref{thm:stopping time} under Assumption \ref{independence}}\label{sec:proof stopping time}
Our goal here is to obtain an analog of Proposition \ref{prop:fixed point} for the random time horizon setting under Assumption \ref{independence}. Once this is established, the asymptotic expansion will follow the same argument as the proof of Theorem \ref{thm:fixed time}.

We use a truncation argument. First let us focus on the finite horizon setting, i.e.~cost function is $h(\mathbf X_T)$. We begin by observing that the operator $\bar{\mathcal{K}}:\mathcal{L}(M)\to\mathcal{L}(M)$ defined as
\begin{equation}
\bar{\mathcal{K}}(L)(x)=\frac{e^{g^L(x)/\alpha}}{E_0[e^{g^L(X)/\alpha}]}
\label{definition K T}
\end{equation}
where $g^L(x)$ is defined in \eqref{optimal L fixed time}, possesses similar contraction properties as the operator $\mathcal{K}$ in \eqref{multivariate K} in the following sense:

\begin{lemma}
With Assumption \ref{boundedness} on the cost function $h(\mathbf{X}_T)$, for sufficiently large $\alpha$, the operator $\bar{\mathcal{K}}:\mathcal{L}(M)\to\mathcal{L}(M)$ is well-defined, closed, and a strict contraction in $\mathcal{L}(M)$ under the metric induced by $\|\cdot\|_\Lambda$. Hence there exists a unique fixed point $L^*\in\mathcal{L}(M)$ that satisfies $\bar{\mathcal{K}}(L)=L$. Moreover, $L^*$ is equal to each identical component of the fixed point of $\mathcal{K}$ defined in \eqref{multivariate K}. \label{lemma:K T}
\end{lemma}

\begin{proof}
We shall utilize our result on the operator $\mathcal{K}$ in Lemma \ref{lemma:multivariate fixed point}. It is easy to check that given $L\in\mathcal{L}(M)$, $\bar{\mathcal{K}}$ acted on $L$ has the same effect as the mapping $K$, defined in \eqref{component K}, acted on $(L,\ldots,L)\in\mathcal{L}(M)^{T-1}$.
In the proof of Lemma \ref{lemma:multivariate fixed point} we have already shown that $K(L_1,\ldots,L_{T-1})$ for any $(L_1,\ldots,L_{T-1})\in\mathcal{L}(M)$ is well-defined, closed, and a strict contraction under $\|\cdot\|_\Lambda$, when $\alpha$ is large enough (or $\beta=1/\alpha$ is small enough in that proof). These properties are inherited immediately to the operator $\bar{\mathcal{K}}$.

Next, note that \eqref{optimal L fixed time} is the fixed point equation associated with $\bar{\mathcal{K}}$. Moreover, we have already shown in Proposition \ref{prop:fixed point} that the same equation governs the fixed point of $\mathcal{K}$, in the sense that the $T-1$ components of its fixed point are all identical and satisfy \eqref{optimal L fixed time}. By the uniqueness property of fixed points, we conclude that the fixed point of $\bar{\mathcal{K}}$ coincides with each identical component of the fixed point of $\mathcal{K}$.
\end{proof}

Now consider a cost function $h(\mathbf{X}_\tau)$ with a random time $\tau$ that satisfies Assumption \ref{independence}. Again let $\beta=1/\alpha>0$ for convenience. We introduce a sequence of truncated random time $\tau\wedge T$, and define $\tilde{\mathcal{K}}_T:\mathcal{L}\to\mathcal{L}$ and $\tilde{\mathcal{K}}:\mathcal{L}\to\mathcal{L}$ as
$$\tilde{\mathcal{K}}_T(L)(x):=\frac{e^{\beta\tilde g^{L,T}(x)}}{E_0[e^{\beta\tilde g^{L,T}(X)}]}$$
and
$$\tilde{\mathcal{K}}(L)(x):=\frac{e^{\beta\tilde g^L(x)}}{E_0[e^{\beta\tilde g^L(X)}]}$$
where
$$\tilde g^{L,T}(x):=\sum_{t=1}^TE_0[h(\mathbf{X}_{\tau\wedge T})\underline{L}_{\tau\wedge T}^t;\tau\wedge T\geq t|X_t=x]$$
and
$$\tilde g^L(x):=\sum_{t=1}^TE_0[h(\mathbf{X}_\tau)\underline{L}_\tau^t;\tau\geq t|X_t=x].$$
Here $\underline{L}_s^t=\prod_{r=1,\ldots,s,r\neq t}L(X_r)$. In other words, $\tilde{\mathcal{K}}_T$ is the map identical to $\tilde{\mathcal{K}}$ except that $\tau$ is replaced by $\tau\wedge T$.

We first need the following proposition:
\begin{proposition}
Suppose Assumption \ref{independence} holds. For $\beta\leq\epsilon$ for some small $\epsilon>0$, both $\tilde{\mathcal{K}}_T$, for any $T\geq1$, and $\tilde{\mathcal{K}}$ are well-defined, closed and strict contractions with the same Lipschitz constant on the space $\mathcal{L}$ equipped with the metric induced by the $\mathcal{L}_1$-norm $\|L-L'\|_1:=E_0|L-L'|$. \label{contraction stopping time}
\end{proposition}

\begin{proof}
We first consider the map $\tilde{\mathcal{K}}$. Recall that by Assumption \ref{independence} we have $|h(\mathbf{X}_\tau)|\leq C$ for some constant $C>0$. Consider
\begin{align}
e^{\beta\sum_{t=1}^\infty E_0[h(\mathbf{X}_\tau)\underline{L}_\tau^t;\tau\geq t|X_t=x]}&\leq e^{C\beta\sum_{t=1}^\infty E_0[\underline{L}_\tau^t;\tau\geq t|X_t=x]} \notag\\
&=e^{C\beta\sum_{t=1}^\infty E_0[\underline{L}_\tau^t;\tau\geq t]}\text{\ \ since $\underline{L}_\tau^tI(\tau\geq t)$ is independent of $X_t$} \notag\\
&=e^{C\beta\sum_{t=1}^\infty P_0(\tau\geq t)}\text{\ \ since $\tau$ is independent of $\{X_t\}_{t\geq1}$} \notag\\
&=e^{C\beta E_0\tau} \notag\\
&<\infty\text{\ \ by Assumption \ref{independence}.} \label{stopping interim 1}
\end{align}
Similarly,
\begin{equation}
e^{\beta\sum_{t=1}^\infty E_0[h(\mathbf{X}_\tau)\underline{L}_\tau^t;\tau\geq t|X_t=x]}\geq e^{-C\beta E_0\tau}>0. \label{stopping interim 2}
\end{equation}
Therefore $\tilde{\mathcal{K}}$ is well-defined and also closed in $\mathcal{L}$. To prove that $\tilde{\mathcal{K}}$ is a contraction, consider, for any $L,L'\in\mathcal{L}$,
\begin{align}
E_0|\tilde{\mathcal{K}}(L)-\tilde{\mathcal{K}}(L')|&=E_0\left|\frac{e^{\beta\tilde g^L(X)}}{E_0[e^{\beta\tilde g^L(X)}]}-\frac{e^{\beta\tilde g^{L'}(X)}}{E_0[e^{\beta\tilde g^{L'}(X)}]}\right| \notag\\
&\leq E_0\left[\sup\left|\frac{1}{\xi_2}\right||e^{\beta\tilde g^L(X)}-e^{\beta\tilde g^{L'}(X)}|+\sup\left|\frac{\xi_1}{\xi_2^2}\right||E_0[e^{\beta\tilde g^L(X)}]-E_0[e^{\beta\tilde g^{L'}(X)}]|\right] \label{stopping interim 3}
\end{align}
by mean value theorem, where $(\xi_1,\xi_2)$ lies in the line segment between $(e^{\beta\tilde g^L(X)},E_0[e^{\beta\tilde g^L(X)}])$ and $(e^{\beta\tilde g^{L'}(X)},E_0[e^{\beta\tilde g^{L'}(X)}])$. By \eqref{stopping interim 1} and \eqref{stopping interim 2}, we have $\xi_1\leq e^{C\beta E_0\tau}$ and $\xi_2\geq e^{-C\beta E_0\tau}$ a.s.. So \eqref{stopping interim 3} is less than or equal to
\begin{equation}
2e^{3C\beta E_0\tau}E_0|e^{\beta\tilde g^L(X)}-e^{\beta\tilde g^{L'}(X)}|\leq2e^{3C\beta E_0\tau}\beta E_0|\xi||\tilde g^L(X)-\tilde g^{L'}(X)| \label{stopping interim 4}
\end{equation}
by mean value theorem again, where $\xi$ lies between $e^{\beta\tilde g^L(X)}$ and $e^{\beta\tilde g^{L'}(X)}$ and hence $\xi\leq e^{C\beta E_0\tau}$ a.s.. Therefore \eqref{stopping interim 4} is further bounded by
\begin{equation}
2e^{4C\beta E_0\tau}\beta\sum_{t=1}^\infty E_0[|\underline{L}_\tau^t-{\underline{L}_\tau'}^t|;\tau\geq t]. \label{stopping interim 5}
\end{equation}
Conditioning on $\tau$, $|\underline{L}_\tau^t-{\underline{L}_\tau'}^t|\leq\sum_{s=1,\ldots,\tau,s\neq t}\dot{\underline{L}}_\tau^{t,s}|L_s-L_s'|$ where $\dot{\underline{L}}_\tau^{t,s}$ is a product of either one $L(X_r)$ or $L'(X_r)$ for $r=1,\ldots,\tau,r\neq t,s$. Since $\tau$ is independent of $\{X_t\}_{t\geq1}$ and $\{X_t\}_{t\geq1}$ are i.i.d.~under $P_0$, we have
$E_0[|\underline{L}_\tau^t-{\underline{L}_\tau'}^t||\tau]\leq(\tau-1)E_0|L-L'|$. Consequently, \eqref{stopping interim 5} is bounded by
$$2e^{4C\beta E_0\tau}\beta\sum_{t=1}^\infty E_0[\tau-1;\tau\geq t]E_0|L-L'|=2e^{4C\beta E_0\tau}\beta E_0[\tau(\tau-1)]E_0|L-L'|\leq wE_0|L-L'|$$
for some $w<1$ when $\beta$ is small enough, since $E_0\tau^2<\infty$ by Assumption \ref{independence}.

Finally, we note that the above arguments all hold with $\tau$ replaced by $\tau\wedge T$, in the same range of $\beta$ and with the same Lipschitz constant $w$. This concludes the proposition.
\end{proof}

Next we show that $\tilde{\mathcal{K}}_T\to\tilde{\mathcal{K}}$ pointwise on $\mathcal{L}$:
\begin{proposition}
Suppose Assumption \ref{independence} is in hold. For any $L\in\mathcal{L}$, we have $\tilde{\mathcal{K}}_T(L)\to\tilde{\mathcal{K}}(L)$ in $\|\cdot\|_1$, uniformly on $\beta\leq\epsilon$ for some small $\epsilon>0$. \label{pointwise convergence}
\end{proposition}

\begin{proof}
Consider
\begin{align}
E_0|\tilde{\mathcal{K}}_T(L)-\tilde{\mathcal{K}}(L)|&=E_0\left|\frac{e^{\beta\tilde g^{L,T}(X)}}{E_0[e^{\beta\tilde g^{L,T}(X)}]}-\frac{e^{\beta\tilde g^{L}(X)}}{E_0[e^{\beta\tilde g^{L}(X)}]}\right| \notag\\
&\leq 2e^{4C\beta E_0\tau}\beta\sum_{t=1}^\infty E_0|h(\mathbf{X}_{\tau\wedge T})\underline{L}_{\tau\wedge T}^tI(\tau\wedge T\geq t)-h(\mathbf{X}_\tau)\underline{L}_\tau^tI(\tau\geq t)| \label{stopping interim 6}
\end{align}
by an argument similar to \eqref{stopping interim 5}. Now consider
\begin{eqnarray*}
&&\sum_{t=1}^\infty E_0|h(\mathbf{X}_{\tau\wedge T})\underline{L}_{\tau\wedge T}^tI(\tau\wedge T\geq t)-h(\mathbf{X}_\tau)\underline{L}_\tau^tI(\tau\geq t)|\\
&=&\sum_{t=1}^\infty E_0[|h(\mathbf{X}_\tau)\underline{L}_\tau^tI(\tau\geq t)-h(\mathbf{X}_\tau)\underline{L}_\tau^tI(\tau\geq t)|;\tau<T]{}\\
&&{}+\sum_{t=1}^\infty E_0[|h(\mathbf{X}_T)\underline{L}_T^tI(T\geq t)-h(\mathbf{X}_\tau)\underline{L}_\tau^tI(\tau\geq t)|;\tau\geq T]\\
&=&\sum_{t=1}^\infty E_0[|h(\mathbf{X}_T)\underline{L}_T^tI(T\geq t)-h(\mathbf{X}_\tau)\underline{L}_\tau^tI(\tau\geq t)|;\tau\geq T]\\
&=&\sum_{t=1}^TE_0[|h(\mathbf{X}_T)\underline{L}_T^t-h(\mathbf{X}_\tau)\underline{L}_\tau^t|;\tau\geq T]+\sum_{t=T+1}^\infty E_0[|h(\mathbf{X}_\tau)|\underline{L}_\tau^t;\tau\geq t]\\
&\leq&2CTP_0(\tau\geq T)+C\sum_{t=T+1}^\infty P_0(\tau\geq t)\text{\ \ for some constant $C>0$}\\
&\to0
\end{eqnarray*}
as $T\to\infty$, since $E_0\tau<\infty$. Hence \eqref{stopping interim 6} converges to 0 uniformly over $\beta\leq\epsilon$ for some small $\epsilon>0$. This concludes the proposition.
\end{proof}

By a simple argument on the continuity of fixed points (see, for example, Theorem 1.2 in \cite{bonsall62}), Proposition \ref{pointwise convergence} implies the following convergence result:
\begin{corollary}
Suppose Assumption \ref{independence} holds. For small enough $\beta$, and letting $L^{(T)}$ and $L^*$ be the fixed points of $\tilde{\mathcal{K}}_T$ and $\tilde{\mathcal{K}}$ respectively, we have $L^{(T)}\stackrel{\mathcal{L}_1}{\to}L^*$.
\end{corollary}

Finally, we show that $L^*$ is the optimal solution to the Lagrangian relaxation of \eqref{max stopping time L}:
\begin{proposition}
Under Assumption \ref{independence}, the fixed point $L^*$ of the operator $\tilde{\mathcal{K}}$ maximizes
$$E_0[h(\mathbf{X}_\tau)\underline{L}_\tau]-\alpha E_0[L\log L]$$
when $\alpha$ is large enough. \label{optimal stopping}
\end{proposition}

\begin{proofof}{Proposition \ref{optimal stopping}}
In the proof we let $C>0$ be a constant, not necessarily the same every time it shows up. To begin, we use the fact that for any fixed $T$, $L^{(T)}$ is the optimal solution to $E_0[h(\mathbf{X}_{\tau\wedge T})\underline{L}_{\tau\wedge T}]-\alpha E_0[L\log L]$, as a direct consequence of Proposition \ref{prop:fixed point}. Hence we have the inequality
\begin{equation}
E_0[h(\mathbf{X}_{\tau\wedge T})\underline{L}_{\tau\wedge T}]-\alpha E_0[L\log L]\leq E_0[h(\mathbf{X}_{\tau\wedge T})\underline{L}_{\tau\wedge T}^{(T)}]-\alpha E_0[L^{(T)}\log L^{(T)}] \label{prelimit}
\end{equation}
for any $L\in\mathcal{L}$ (since $h$ is bounded we can merely replace $\mathcal{L}(M)$ by $\mathcal{L}$, i.e.~putting $M=\infty$, in Proposition \ref{prop:fixed point}). Here \eqref{prelimit} holds for any $T\geq1$ for $\alpha$ uniformly large (the uniformity can be verified using Proposition \ref{contraction stopping time} and repeating the argument in the proof of Lemma \ref{lemma:multivariate convergence}, noting that $\tau\wedge T\leq\tau$ a.s.). Our main argument consists of letting $T\to\infty$ on both sides of \eqref{prelimit}.

We first show that, for any $L\in\mathcal{L}$, the first term on the left hand side of \eqref{prelimit} converges to $E_0[h(\mathbf{X}_\tau)\underline{L}_\tau]$. Consider
\begin{eqnarray*}
&&|E_0[h(\mathbf{X}_{\tau\wedge T})\underline{L}_{\tau\wedge T}]-E_0[h(\mathbf{X}_\tau)\underline{L}_\tau]|\\
&=&\left|\sum_{t=1}^\infty E_0[h(\mathbf{X}_t)\underline{L}_t]P_0(\tau\wedge T=t)-\sum_{t=1}^\infty E_0[h(\mathbf{X}_t)\underline{L}_t]P_0(\tau=t)\right|\\
&\leq&C\sum_{t=1}^{T-1}|P_0(\tau\wedge T=t)-P_0(\tau=t)|+C|P_0(\tau\geq T)-P_0(\tau=T)|+C\sum_{t=T+1}^\infty P_0(\tau=t)\\
&=&CP_0(\tau>T)+CP_0(\tau\geq T+1)\\
&\to&0
\end{eqnarray*}
as $T\to\infty$, since $E_0\tau<\infty$. Hence the left hand side of \eqref{prelimit} converges to $E_0[h(\mathbf{X}_\tau)\underline{L}_\tau]-\alpha E_0[L\log L]$ for any $L\in\mathcal{L}$. Now consider the right hand side. For the first term, consider
\begin{eqnarray}
&&|E_0[h(\mathbf{X}_{\tau\wedge T})\underline{L}_{\tau\wedge T}^{(T)}]-E_0[h(\mathbf{X}_\tau)\underline{L}_\tau^*]| \notag\\
&=&\left|\sum_{t=1}^\infty E_0[h(\mathbf{X}_t)\underline{L}_t^{(T)}]P_0(\tau\wedge T=t)-\sum_{t=1}^\infty E_0[h(\mathbf{X}_t)\underline{L}_t^*]P_0(\tau=t)\right| \notag\\
&\leq&C\sum_{t=1}^{T-1}E_0|\underline{L}_t^{(T)}-\underline{L}_t^*|P_0(\tau=t)+2C(P_0(\tau\geq T)+P(\tau=T))+C\sum_{t=T+1}^\infty P_0(\tau=t) \notag\\
&\leq&C\sum_{t=1}^TtP_0(\tau=t)E_0|L^{(T)}-L^*|+2C(P_0(\tau\geq T)+P(\tau=T))+CP_0(\tau\geq T+1){} \notag\\
&&{}\text{\ \  by the argument following \eqref{stopping interim 5}} \notag\\
&=&CE_0[\tau;\tau\leq T]E_0|L^{(T)}-L^*|+2C(P_0(\tau\geq T)+P(\tau=T))+CP_0(\tau\geq T+1) \notag\\
&\to&0. \label{stopping interim 7}
\end{eqnarray}
Moreover, for the second term in \eqref{prelimit}, write
$$E_0[L^{(T)}\log L^{(T)}]=\beta E_0[h(\mathbf{X}_{\tau\wedge T})\underline{L}_{\tau\wedge T}^{(T)}]-\log E_0[e^{\beta\tilde g^{L^{(T)},T}(X)}]$$
and
$$E_0[L^*\log L^*]=\beta E_0[h(\mathbf{X}_\tau)\underline{L}_\tau^*]-\log E_0[e^{\beta\tilde g^{L^*}(X)}]$$
by the definition of the fixed points for $\tilde{\mathcal{K}}_T$ and $\tilde{\mathcal{K}}$. To prove that $E_0[L^{(T)}\log L^{(T)}]\to E_0[L^*\log L^*]$, we have to show that $E_0[h(\mathbf{X}_{\tau\wedge T})\underline{L}_{\tau\wedge T}^{(T)}]\to E_0[h(\mathbf{X}_\tau)\underline{L}_\tau^*]$, which is achieved by \eqref{stopping interim 7}, and that $\log E_0[e^{\beta\tilde g^{L^{(T)},T}(X)}]\to\log E_0[e^{\beta\tilde g^{L^*}(X)}]$, which we will show as follows. Consider
\begin{eqnarray}
&&|\log E_0[e^{\beta\tilde g^{L^{(T)},T}(X)}]-\log E_0[e^{\beta\tilde g^{L^*}(X)}]| \notag\\
&\leq&e^{C\beta E_0\tau}|E_0[e^{\beta\tilde g^{L^{(T)},T}(X)}]-E_0[e^{\beta\tilde g^{L^*}(X)}]| \text{\ \ by mean value theorem and the bound in \eqref{stopping interim 2}}\notag\\
&\leq&e^{2C\beta E_0\tau}\beta\sum_{t=1}^\infty E_0|h(\mathbf{X}_{\tau\wedge T}){\underline{L}_{\tau\wedge T}^{(T)}}^tI(\tau\wedge T\geq t)-h(\mathbf{X}_\tau){\underline{L}_\tau^*}^tI(\tau\geq t)| \label{stopping interim 8}
\end{eqnarray}
where ${\underline{L}_{\tau\wedge T}^{(T)}}^t=\prod_{\substack{s=1,\ldots,\tau\wedge T\\s\neq t}}L(X_s)^{(T)}$ and ${\underline{L}_{\tau\wedge T}^*}^t=\prod_{\substack{s=1,\ldots,\tau\wedge T\\s\neq t}}L(X_s)^*$, by arguments similar to \eqref{stopping interim 3}-\eqref{stopping interim 5}. Now
\begin{eqnarray}
&&\sum_{t=1}^\infty E_0|h(\mathbf{X}_{\tau\wedge T}){\underline{L}_{\tau\wedge T}^{(T)}}^tI(\tau\wedge T\geq t)-h(\mathbf{X}_\tau){\underline{L}_\tau^*}^tI(\tau\geq t)| \notag\\
&=&\sum_{t=1}^\infty E_0[|h(\mathbf{X}_\tau){\underline{L}_\tau^{(T)}}^tI(\tau\geq t)-h(\mathbf{X}_\tau){\underline{L}_\tau^*}^tI(\tau\geq t)|;\tau<T]{} \notag\\
&&{}+\sum_{t=1}^\infty E_0[|h(\mathbf{X}_T){\underline{L}_T^{(T)}}^tI(T\geq t)-h(\mathbf{X}_\tau){\underline{L}_\tau^*}^tI(\tau\geq t)|;\tau\geq T]. \label{stopping interim 9}
\end{eqnarray}
Note that the first term in \eqref{stopping interim 9} is bounded by
\begin{eqnarray*}
C\sum_{t=1}^\infty E_0[|{\underline{L}_\tau^{(T)}}^t-{\underline{L}_\tau^*}^t|;\tau\geq t,\tau<T]&\leq& C\sum_{t=1}^{T-1}E_0[\tau-1;t\leq\tau<T]E_0|L^{(T)}-L^*|{}\\
&&{}\text{\ \  by the argument following \eqref{stopping interim 5}}\\
&=&CE_0[\tau(\tau-1);\tau<T]E_0|L^{(T)}-L^*|\\
&\to&0
\end{eqnarray*}
since $E_0\tau^2<\infty$. The second term in \eqref{stopping interim 9} can be written as
\begin{eqnarray*}
&&\sum_{t=1}^TE_0[|h(\mathbf{X}_T){\underline{L}_T^{(T)}}^t-h(\mathbf{X}_\tau){\underline{L}_\tau^*}^tI(\tau\geq t)|;\tau\geq T]+\sum_{t=T+1}^\infty E_0[|h(\mathbf{X}_\tau)|{\underline{L}_\tau^*}^t;\tau\geq t]\\
&\leq&2CTP_0(\tau\geq T)+C\sum_{t=T+1}^\infty P_0(\tau\geq t)\\
&\to&0
\end{eqnarray*}
since $E_0\tau<\infty$. We therefore prove that \eqref{stopping interim 8} converges to 0 and the right hand side of \eqref{prelimit} converges to $E_0[h(\mathbf{X}_\tau)\underline{L}_\tau^*]-\alpha E_0[L^*\log L^*]$. This concludes the proposition.
\end{proofof}

\appendix
\section{Sufficiency Theorem}
\begin{thm}[a.k.a. Chapter 8, Theorem 1 in \cite{luenberger69}]
Consider $\phi(\cdot):\mathcal L\to\mathbb R$ and $\mathcal C$ a subset of $\mathcal L$. Suppose there is an $\alpha^*$, with $\alpha^*\geq0$, and an $L^*\in\mathcal{C}$ such that
\begin{equation}
\phi(L^*)-\alpha^*E_0[L^*\log L^*]\geq \phi(L)-\alpha^*E_0[L\log L] \label{optimality}
\end{equation}
for all $L\in\mathcal{C}$. Then $L^*$ solves
$$\begin{array}{ll}
\max&\phi(L)\\
\text{subject to}&E_0[L\log L]\leq E_0[L^*\log L^*]\\
&L\in\mathcal{C}.
\end{array}$$
\label{nonconvex}
\end{thm}

For the proof for Theorem \ref{thm:basic}, $\mathcal C$ is chosen as $\mathcal L$ and $\phi(L)=E_0[h(X)L]$. For Theorem \ref{thm:fixed time}, $\mathcal C$ is chosen as $\mathcal L(M)$ and $\phi(L)=E_0[h(\mathbf X_T)\underline L_T]$, and for Theorem \ref{thm:stopping time} under Assumption \ref{independence}, $\mathcal C$ as $\mathcal L$ and $\phi(L)=E_0[h(\mathbf X_\tau)\underline L_\tau]$.

\bibliographystyle{abbrv}
\bibliography{bibliography}

\end{document}